\documentclass[11pt]{amsart}
\usepackage[utf8x]{inputenc}
\usepackage[english]{babel}
\usepackage[T1]{fontenc}
 
\usepackage{graphicx}
\usepackage{caption}
\usepackage{subcaption}

\usepackage{amssymb,amsmath}
\usepackage[hidelinks]{hyperref}
\usepackage{indentfirst}
\usepackage{enumerate,amsmath,amssymb, mathrsfs,mathtools}
\usepackage{appendix}
\usepackage{latexsym}
\usepackage{url}
\usepackage{color}
\usepackage{accents}
\usepackage{setspace}
\usepackage{pdfpages}
\usepackage{stmaryrd}
\usepackage{amsrefs}

\usepackage[margin=2.5cm]{geometry}
\allowdisplaybreaks

\newcommand{\hcirc}{\accentset{\circ}{h}}

\BibSpec{article}{%
	+{}{\PrintAuthors} {author}
	+{,}{ } {title}
	+{, }{\textit } {journal}
	+{}{ \parenthesize} {date}
		+{,  }{no. } {volume}
	+{,}{ } {pages}
	+{,}{ } {note}
}

\ExplSyntaxOn

\NewExpandableDocumentCommand{\gobblefirst}{m}
{
	\tl_tail:n { #1 }
}

\ExplSyntaxOff
\BibSpec{book}{%
	+{}{\PrintAuthors}  {author}
	+{. }{}{title}
	+{,}{ }{series}
	+{,}{ vol.~}{volume}
	+{. }{\textit}{publisher}
	+{,}{ ISBN \gobblefirst}{isbn}
}
\BibSpec{collection.article}{%
	+{}{\PrintAuthors}{author}
	+{, }{}{title}
	+{, }{\textit}{booktitle}
	+{, }{ \DashPages}{pages}
	+{,}{ }{series}
	+{, }{}{volume}
	+{, }{\textit}{publisher}
	+{,}{ }{date}
}

\mathcode`l="8000
\begingroup
\makeatletter
\lccode`\~=`\l
\DeclareMathSymbol{\lsb@l}{\mathalpha}{letters}{`l}
\lowercase{\gdef~{\ifnum\the\mathgroup=\m@ne \ell \else \lsb@l \fi}}%
\endgroup

\def\XXint#1#2#3{{\setbox0=\hbox{$#1{#2#3}{\int}$ }
		\vcenter{\hbox{$#2#3$ }}\kern-.6\wd0}}

\newtheorem{prop}{Proposition}
\newtheorem{thm}[prop]{Theorem}
\newtheorem{lem}[prop]{Lemma}
\newtheorem{coro}[prop]{Corollary}

\newtheorem{rema}[prop]{Remark}

\newtheorem{conj}[prop]{Conjecture}

\title[Limits of isoperimetric surfaces]{Schoen's conjecture for limits of isoperimetric surfaces}
\author{Michael Eichmair}
\address{
	\textnormal{Michael Eichmair \newline  \indent
		University of Vienna \newline \indent
		Faculty of Mathematics  \newline \indent
		Oskar-Morgenstern-Platz 1 \newline \indent
		1090 Vienna, 	Austria  \newline\indent 
		 \href{https://orcid.org/0000-0001-7993-9536}{https://orcid.org/0000-0001-7993-9536} \newline\indent	
		 \href{mailto:michael.eichmair@univie.ac.atm}{michael.eichmair@univie.ac.at}}
}

\author{Thomas Koerber}
\address{\textnormal{Thomas Koerber  \newline \indent
		University of Vienna \newline \indent
		Faculty of Mathematics  \newline \indent
		Oskar-Morgenstern-Platz 1 \newline \indent 1090 Vienna,	Austria \newline\indent 
		 \href{https://orcid.org/0000-0003-1676-0824}{https://orcid.org/0000-0003-1676-0824} \newline \indent
		  \href{mailto:thomas.koerber@univie.ac.atm}{thomas.koerber@univie.ac.at}}
}

\begin{document}

	\date{\today}
	\onehalfspacing
	\begin{abstract}	
		Let $(M,g)$ be an $n$-dimensional asymptotically flat Riemannian manifold with nonnegative scalar curvature that admits a noncompact area-minimizing hypersurface $\Sigma \subset M$. In the case where $n = 3$, O.~Chodosh and the first-named author have proven that $(M, g)$ is necessarily isometric to Euclidean space, confirming a conjecture of R.~Schoen. In this paper, we extend this result to dimension $3 < n \leq 7$ provided that $\Sigma$ arises as a limit of isoperimetric surfaces. By contrast, we prove that when $3 < n \leq 7$, there is no such  result  for general noncompact area-minimizing $\Sigma \subset M$, even when additional assumptions on the stability of $\Sigma$ are imposed.
			\end{abstract}
	\maketitle 
	\section{Introduction}	
	Throughout, we assume that $(M,g)$ is a connected complete Riemannian manifold. \\ \indent 
The following conjecture of R.~Schoen is related to his proof of the positive mass theorem with S.-T.~Yau in \cites{SchoenYau}.

\begin{conj}[Cp.~{\cite[p.~979]{CCE}}] \label{conjecture} 
	Let $(M,g)$ be an asymptotically flat Riemannian manifold of dimension $ n=3$  with nonnegative scalar curvature. Suppose that there exists a noncompact area-minimizing boundary $\Sigma =\partial \Omega$, $\Omega\subset M$. Then $(M,g)$ is isometric to flat $\mathbb{R}^3$.
\end{conj}

The background on asymptotically flat Riemannian manifolds, area-minimizing boundaries, and isoperimetric regions used in this paper is recalled in Appendix \ref{adm appendix} and Appendix \ref{iso appendix}. Our convention is that the boundary of an asymptotically flat Riemannian manifold $(M,g)$ is either empty or a closed minimal surface that is outermost in the sense that that every closed minimal hypersurface in $(M, g)$ is contained in $\partial M$.   \\ \indent 
Conjecture \ref{conjecture} has been confirmed by O.~Chodosh and the first-named author \cite[Theorem 1.6]{CCE}. A natural way in which noncompact area-minimizing boundaries arise is as  limits of isoperimetric surfaces. The goal of this paper is to extend Conjecture \ref{conjecture} to higher dimensions in this setting.

\begin{thm}
	\label{main result 2} Let $(M,g)$ be a Riemannian manifold of dimension $3<n\leq 7$ that is asymptotically flat of rate $\tau>n-3$ and has nonnegative scalar curvature. Suppose that there exist a noncompact area-minimizing boundary $\Sigma =\partial \Omega$, $\Omega\subset M$, and isoperimetric regions $\Omega_1,\,\Omega_2,\hdots\subset M$  such that $\Omega_k\to \Omega$ locally smoothly. Then $(M,g)$ is isometric to flat $\mathbb{R}^n$.
\end{thm}

O.~Chodosh, Y.~Shi, H.~Yu, and the first-named author have shown that in asymptotically flat Riemannian 3-manifolds with nonnegative scalar curvature and positive mass, there is a unique isoperimetric region for every given sufficiently large amount of volume and that these large isoperimetric regions are close to centered coordinate balls in the chart at infinity; see \cite[Theorem 1.1]{CESH}. An alternative proof of this result with a different condition on the scalar curvature was subsequently found by H.~Yu; see \cite[Theorem 1.6]{yu}. As a step towards the characterization of large isoperimetric regions in asymptotically flat Riemannian manifolds of dimension $3 < n \leq 7$, Corollary \ref{main result} shows that the (unique) large components of the boundaries of such regions necessarily diverge.
\begin{coro}\label{main result} Let $(M,g)$ be a Riemannian manifold of dimension $3<n\leq 7$ that is asymptotically flat of rate $\tau>n-3$ and has nonnegative scalar curvature and positive mass. Let $K\subset M$ be a compact set such that $K\cap \partial M=\emptyset$. Suppose that there are isoperimetric regions $\Omega_1,\,\Omega_2,\hdots\subset M$ with $|\Omega_k|\to\infty$. Then, for all $k$ sufficiently large, either $K\subset \Omega_k$  or $K\cap \Omega_k=\emptyset$.
\end{coro}

\begin{rema}
	The assumption that $\tau>n-3$ is required in our proof in two different places. First, in the proof of Proposition \ref{gamma limit lemma}, it is used to estimate the difference between the second variation of area of $\partial \Omega_k$ and of $\partial \Omega$. Second, the assumption guarantees that coordinate hyperplanes in the end of $(M,g)$ are asymptotically flat with mass zero. This is crucial in the proof of Proposition \ref{flat flat}. Our methods give no clue as to whether Theorem \ref{main result 2} holds for all $\tau > (n-2)/2$. If true, then a different strategy of proof would be required.
\end{rema}

\subsection*{Outline of our arguments}
 Let $(M,g)$ be an asymptotically flat Riemannian manifold of dimension $3\leq n\leq 7$ with nonnegative scalar curvature. Suppose that $\Sigma=\partial \Omega$ is a noncompact area-minimizing boundary. In particular, for every open set $U \Subset M$ and every smooth variation $\{\Sigma(s)\}_{|s| < \epsilon}$ of $\Sigma = \Sigma(0)$ with compact support in $U$,
 $$ \frac{d}{ds}\bigg|_{s=0} |U\cap \Sigma(s)| = 0\qquad  \text{ and }\qquad \frac{d^2}{ds^2}\bigg|_{s=0} |U\cap \Sigma(s) | \geq 0.
 $$
 Equivalently, the mean curvature of $\Sigma$ vanishes and the stability inequality
  $$
 \int_{\Sigma} (|h|^2+Ric(\nu,\nu))\,f^2\,\mathrm{d}\mu\leq \int_{\Sigma} |\nabla f|^2\,\mathrm{d}\mu
 $$
 holds for all $f\in C^\infty_c(\Sigma)$. Here, $\mathrm{d}\mu$ is the area element, $\nabla$ the covariant derivative, $\nu$ the outward normal, and $h$ the second fundamental form, all with respect to $\Sigma$. $Ric$ denotes the Ricci tensor of $(M,g)$. We say that $\Sigma$ is stable with respect to asymptotically constant variations if, in addition, 
 \begin{align} \label{stable wrt acv} 
 \int_{\Sigma} (|h|^2+Ric(\nu,\nu))\,(1+f)^2\,\mathrm{d}\mu\leq \int_{\Sigma} |\nabla f|^2\,\mathrm{d}\mu
 \end{align} 
 for all $f\in C^\infty_c(\Sigma)$.\\
 \indent The proof of the positive mass theorem \cite[Theorem 4.2]{montecatini} shows that an asymptotically flat area-minimizing boundary that has mass zero and  is stable with respect to asymptotically constant variations is isometric to flat $\mathbb{R}^{n-1}$ and totally geodesic. Moreover, the scalar curvature of $(M,g)$ vanishes along such a boundary; see Lemma \ref{tilde flat}. An important ingredient in the proof of Conjecture \ref{conjecture} in \cite[Theorem 1.6]{CCE}  specific to three dimensions is that every noncompact area-minimizing boundary is stable with respect to asymptotically constant variations; see, e.g., \cite[p.~54]{SchoenYau}. Using this, O.~Chodosh and the first-named author have shown that $(M,g)$ is foliated by noncompact area-minimizing boundaries. The construction of these boundaries is based on solving Plateau problems with respect to a carefully chosen local perturbation of the metric $g$ and inspired by the proof of a conjecture of J.~Milnor due to G.~Liu \cite{liu}. An adaptation of an argument by M.~Anderson and L.~Rodr\'iguez \cite{andersonrodriguez} shows that the curvature tensor of $(M,g)$ vanishes along each leaf of this foliation and hence on all of $M$. \\ \indent
 As observed by R.~Schoen and S.-T.~Yau in \cite{pnas}, the situation is markedly different in higher dimensions. Indeed, every asymptotically flat Riemannian manifold $(M,g)$ of dimension $3<n\leq 7$ admits infinitely many noncompact area-minimizing hypersurfaces; see Remark \ref{schoenyaucarlotto}. Building on the gluing technique devised by A.~Carlotto and R.~Schoen \cite{carlottoschoen} and further developed by Y.~Mao and Z.~Tao \cite{maotao}, our first result shows that such $(M,g)$ can be chosen to have nonnegative scalar curvature and positive mass and such that some of these area-minimizing hypersurfaces satisfy \eqref{stable wrt acv}. 
\begin{thm} \label{counterexample 2}
 Let $3 < n \leq 7$. There exists a Riemannian manifold $(M,g)$ of dimension $n$ that is asymptotically flat of rate $\tau=n-2$ with nonnegative scalar curvature and positive mass that contains infinitely many mutually disjoint noncompact area-minimizing hypersurfaces,  all of which are stable with respect to asymptotically constant variations. 
\end{thm}
\begin{rema}
	Theorem \ref{counterexample 2} shows that an extension of Conjecture \ref{conjecture} to higher dimensions is not true without imposing any further global assumptions. By contrast,  the conclusion of  Conjecture \ref{conjecture} has been verified by A.~Carlotto \cite[Theorem 1 and Theorem 2]{Carlottocalcvar} in the case where $3\leq n\leq 7$ under the additional assumptions that $(M,g)$ is asymptotic to Schwarzschild and that $\Sigma$ is stable with respect to asymptotically constant variations.  In this case, Lemma \ref{tilde flat}  yields an immediate contradiction once $\Sigma$ is shown to be asymptotically flat with  mass zero; cf. Lemma \ref{vanishing mass}.
	
\end{rema}
\begin{rema}
	O.~Chodosh and D.~Ketover have shown in \cite{chodoshketover} that in every complete asymptotically flat Riemannian 3-manifold $(M, g)$ that does not contain closed embedded minimal surfaces, through every point, there exists a properly embedded minimal plane; see also the subsequent improvement due to L.~Mazet and H.~Rosenberg in \cite{mazetrosenberg}. Note that if the scalar curvature of $(M, g)$ is nonnegative, none of these planes are area-minimizing unless $(M, g)$ is flat $\mathbb{R}^3$.
\end{rema}
In the proof of \cite[Theorem 4.2]{montecatini}, R.~Schoen and S.-T.~Yau have shown that if $(M,g)$ is asymptotic to Schwarzschild
\begin{align} \label{schwarzschild}  
g_S=\left(1+\frac{m}{2}\,|x|_{\bar g}^{2-n}\right)^{\frac{4}{n-2}}\,\bar g
\end{align}  with negative mass $m<0$, then $(M,g)$ contains a noncompact area-minimizing boundary that is stable with respect to asymptotically constant variations. Here, $\bar g$ is the Euclidean metric. This boundary is obtained as the limit of solutions to the Plateau problem with prudently chosen boundaries. Our starting point is the following complementary consideration. If $\Sigma=\partial \Omega$  arises as the limit of isoperimetric surfaces, then we expect $\Sigma$  to be stable with respect to asymptotically constant variations.  \\ \indent 
We now describe the proof of Theorem \ref{main result 2}. Suppose that $3<n\leq 7$. A first difficulty not present in the case where $n=3$ is to show that  $\Sigma$ is asymptotically flat with  mass zero. This is complicated by the fact that  $\Sigma$ is not known to be stable with respect to asymptotically constant variations at this point. By contrast, this stability is an additional assumption in the work of A.~Carlotto \cite{Carlottocalcvar}. To remedy this, we prove the explicit estimate
\begin{align} \label{intro area est} 
1-O(r^{-\frac{\tau}{n-1}})\leq \frac{|B_r\cap \Sigma|}{\omega_{n-1}\,r^{n-1}}\leq 1+O(r^{-\tau});
\end{align} 
see Lemma \ref{upper area bound} and Lemma \ref{lemma lower area bound}.
Here, $B_r$, $r>1$, is the bounded region in $(M,g)$  whose boundary corresponds to $S^{n-1}_r(0)=\{x\in\mathbb{R}^n:|x|_{\bar g}=r\}$ in the chart at infinity and $\omega_{n-1}$ the Euclidean area of the $(n-1)$-dimensional unit ball $B^n_1(0)$. The proof of \eqref{intro area est} is based on the monotonicity formula applied to carefully chosen, off-centered balls. We then use \eqref{intro area est} to prove a precise asymptotic expansion for $\Sigma$; see Proposition \ref{graph estimate}. Using that $\tau>n-3$, it follows that $\Sigma$ is asymptotically flat with  mass zero; see Lemma \ref{vanishing mass}. We remark that these arguments work for any stable properly embedded noncompact minimal hypersurface with $r^{1-n} \, |B_r\cap \Sigma|=O(1)$ for $r >1$.\\ \indent
Next, we assume that $\Sigma=\partial \Omega$ where $\Omega$ is the limit of large isoperimetric regions $\Omega_1,\,\Omega_2,\hdots\subset M$ with $|\Omega_k|\to\infty$ and prove that $\Sigma$ is stable with respect to asymptotically constant variations. To this end, we consider the second variation of area of $\Omega_k$ with respect to a suitable Euclidean translation that is corrected to be volume-preserving. The stability with respect to asymptotically constant variations then follows by passing to the limit $k\to\infty$, using the asymptotic expansion for $\Sigma$ obtained in Proposition \ref{graph estimate}, the assumption that $\tau>n-3$, and the integration by parts formula in Lemma  \ref{euclidean lemma}; see Proposition \ref{strictly stable}. The arguments from \cite{montecatini} then show that $\Sigma$ is isometric to flat $\mathbb{R}^{n-1}$ and totally geodesic; see Proposition \ref{flat flat}.  \\ \indent Finally, given any point $p\in M$, we construct a new noncompact area-minimizing boundary $\Sigma_p\subset M$ that passes through $p$; see Proposition \ref{many area minimizing}. In view of Remark \ref{schoenyaucarlotto} and different from the situation in \cite{CCE}, we need to ensure that $\Sigma_p$ is again stable with respect to asymptotically constant variations. To this end, we construct suitable local perturbations of the metric $g$ in Lemma \ref{metric perturbation} and obtain $\Sigma_p$ as the limit of large isoperimetric regions with respect to these perturbations. A crucial ingredient in the construction of $\Sigma_p$ is a result from \cite{CCE}, stated here as Lemma \ref{existence iso}, namely: Asymptotically flat Riemannian manifolds of positive mass admit isoperimetric regions of every sufficiently large volume. Although the area-minimizing boundaries obtained in our construction do not necessarily form a foliation, we show how to adapt the techniques developed in \cites{CCE,andersonrodriguez,liu} to conclude that the curvature tensor of $(M,g)$ vanishes along each of these boundaries. This completes the proof of Theorem \ref{main result 2}.
\subsection*{Outline of related results} J. Metzger and the first-named author \cite{Eichmair-Metzer:2012} have observed that the existence of area-minimizing boundaries in asymptotically flat Riemannian manifolds is related to the positioning of large isoperimetric regions. In particular, they  show that, in asymptotically flat Riemannian 3-manifolds, the existence of large isoperimetric regions that do not diverge  is not compatible with positive scalar curvature; see \cite[Theorem 1.5]{Eichmair-Metzer:2012}. In subsequent work \cite[Theorem 1.1]{eichmairmetzgerinvent}, they have shown that, if $(M,g)$ is asymptotic to Schwarzschild of dimension $n \geq 3$, large isoperimetric regions are unique and geometrically close to centered coordinate balls. In particular, Theorem \ref{main result 2} holds in all dimensions if $(M,g)$ is assumed to be asymptotic to Schwarzschild. 
\subsection*{Acknowledgments}
The authors thank the referees for their careful reading of our work. Their comments have led to several corrections and other improvements. They are grateful to Richard Schoen for pointing out his observation with S.-T.~Yau, noted in \cite{pnas}, that asymptotically flat Riemannian manifolds of dimension $3 < n \leq 7$ admit area-minimizing hypersurfaces asymptotic to any given hyperplane. This research was funded in part by the Austrian Science Fund (FWF) [10.55776/PAT9828924, 10.55776/M3184, 10.55776/Y963]. For open access purposes, the authors have applied a CC BY public copyright license to an author accepted manuscript version of this paper.

\section{Asymptotic behavior of area-minimizing surfaces}
In this section, we assume that $g$ is a Riemannian metric on $\mathbb{R}^n$  such that
\begin{align} \label{AF section 2}
|g-\bar g|_{\bar g}+|x|_{\bar g}\,|\bar D g|_{\bar g}+|x|^2_{\bar g}\,|\bar D^2 g|_{\bar g}=O(|x|_{\bar g}^{-\tau})
\end{align}
where $3\leq  n\leq 7$ and $0<\tau< n-2$. Here, $\bar g$ denotes the Euclidean metric and $\bar D$ the Euclidean derivative. Geometric quantities are computed with respect to $g$ unless indicated otherwise. \\
\indent Let $\Sigma \subset \mathbb{R}^n$ be a noncompact two-sided properly embedded hypersurface with $\partial \Sigma=\emptyset$. We assume that $$\{x\in \mathbb{R}^n:|x|_{\bar g}>1\}\cap\Sigma$$ is a stable minimal surface and that
\begin{align} \label{area growth polynomial}  
\limsup_{r\to\infty} \frac{|B^n_r(0)\cap\Sigma|_{\bar g}}{\omega_{n-1}\,r^{n-1}}<\infty.
\end{align} 
 \indent 
The goal of this section is to prove the following result.
 \begin{prop} \label{graph estimate} There exist $r_0>2,$ an integer $N\geq 1$, numbers $a_1,\hdots,a_N\in\mathbb{R}$, functions $\psi_1,\hdots, \psi_N\in C^\infty(\mathbb{R}^{n-1}),$ and a rotation $S\in SO(n)$ such that 
$$
S(\Sigma\setminus B^n_{r_0}(0))\subset \bigcup_{i=1}^N\{(y,\psi_i(y)):y\in\mathbb{R}^{n-1}\}.
$$
Moreover, for all $0<\varepsilon<\tau/2$ and $i=1,\dots,N$,
 		\begin{align} \label{graph estimate optimal 2} 
 			|y|^{-1}_{\bar g}\,|\psi_i(y)-a_i|+|\bar\nabla \psi_i|_{\bar g}+|y|_{\bar g}\,|\bar\nabla^2 \psi_i |_{\bar g}=O(|y|_{\bar g}^{-\tau+\varepsilon }).
 		\end{align} 
 \end{prop}
\begin{rema} In the case where $g$ is asymptotic to the Schwarzschild metric \eqref{schwarzschild}, Proposition \ref{graph estimate} has been proven by A.~Carlotto in \cite[Lemma 18]{Carlottocalcvar} under the additional assumption that $\Sigma$ is stable with respect to asymptotically constant variations in the sense of \eqref{stable wrt acv}.  There, a version of Corollary \ref{hyperplane uniqueness} is obtained using techniques developed by L.~Simon \cite[Theorem 5.7]{simonisolated}; see \cite[p.~10]{Carlottocalcvar}. A version of  estimate \eqref{hoelder decay} is obtained as a consequence  of the stability with respect to asymptotically constant variations; see \cite[pp.~16-18]{Carlottocalcvar}. Here, we provide a new proof based on the monotonicity formula of these results in \cite{Carlottocalcvar}   that does not require the assumption of stability with respect to asymptotically constant variations. 
\end{rema}

We first reduce the proof of Proposition \ref{graph estimate} to the case where $\Sigma$ has only one end.
\begin{lem} \label{components reduction}
	There exists $r_0>1$ and an integer $N\geq 1$ such that $\Sigma\setminus  B^n_{r_0}(0)$ has $N$ connected components $\Sigma_1,\hdots,\Sigma_N\subset\mathbb{R}^n$ each satisfying
	\begin{align} \label{multiplicity one}
	\lim_{r\to\infty}\frac{|B^n_{r}(0)\cap \Sigma_i|_{\bar g}}{\omega_{n-1}\,r^{n-1}}=1.
	\end{align} 
\end{lem}
\begin{proof}
By the work of R.~Schoen and L.~Simon \cite[Corollary 1]{schoensimon},  $h=O(|x|^{-1}_{\bar g})$ where $h$ is the second fundamental form of $\Sigma$ with respect to a choice of unit normal. In conjunction with \eqref{area growth polynomial}, the assumption that $\Sigma$ is noncompact, and the classification of stable minimal cones in $\mathbb{R}^n$ by J.~Simons \cite[\S6]{Simons}, it follows that for each sequence $r_1,\,r_2,\ldots>0$ with $r_k\to\infty$, $r_k^{-1}\,\Sigma$ converges to a hyperplane locally smoothly in $\mathbb{R}^n\setminus\{0\}$   with multiplicity $N\geq 1$. Moreover, 
$$
N \leq \limsup_{r\to\infty} \frac{|B^n_r(0)\cap\Sigma|_{\bar g}}{\omega_{n-1}\,r^{n-1}}.
$$
In particular, $\Sigma$ intersects $S^{n-1}_r(0)$ transversely for all $r>1$ sufficiently large. It follows that there is $r_0>1$ such that the number of components of $\Sigma\setminus B^n_{r}(0)$ is finite and constant for all $r>r_0$. Let $\Sigma_1$ be a component of $\Sigma\setminus B^n_{r_0}(0)$. Applying the same argument with $\Sigma$ replaced by $\Sigma_1$ and using that  $\Sigma_1\setminus B^n_r(0)$ is connected for all $r>r_0$, we see that $r_k^{-1}\,\Sigma_1$ converges  to a hyperplane locally smoothly in $\mathbb{R}^n\setminus \{0\}$ with multiplicity one for each sequence $r_1,\,r_2,\ldots>0$ with $r_k\to\infty$. The assertion follows.

\end{proof}
In view of Lemma \ref{components reduction}, we may and will assume that
\begin{align} \label{area bound at infinity}
	\lim_{r\to\infty}\frac{|B^n_{r}(0)\cap \Sigma|_{\bar g}}{\omega_{n-1}\,r^{n-1}}=1
\end{align}
in the proof of Proposition \ref{graph estimate}. By perturbing $\{x\in \mathbb{R}^n:|x|_{\bar g}<1\}\cap \Sigma$ if necessary, we may and will assume that 
\begin{align}\label{inner radius}
\inf\{|x|_{\bar g}:x\in \Sigma\}>0.
\end{align}
We also record the following by-product of the proof of Lemma \ref{components reduction}.
\begin{lem} \label{blowdown at infinity are planes}
	Let $r_1,\,r_2,\ldots>1$ be numbers with $r_k\to\infty$. Then, passing to a subsequence, $r_k^{-1}\,\Sigma$ converges locally smoothly in $\mathbb{R}^n\setminus \{0\}$  with multiplicity one to a hyperplane through the origin. 
\end{lem}
 \indent 
We collect basic properties of $\Sigma$ in Lemma \ref{mean curvature estimate} and Lemma \ref{area estimate sigma bar}.
\begin{lem} There holds, as $|x|_{\bar g}\to\infty$,  \label{mean curvature estimate}
\begin{align} \label{H estimate} 
|x|_{\bar g}\,	|\bar H|+|x|^2_{\bar g}\,|\bar\nabla \bar H|_{\bar g}=O(|x|_{\bar g}^{-\tau})
\end{align}
and
\begin{align} \label{h estimate}
	|x|_{\bar g}\,|\bar h|_{\bar g}+|x|^2_{\bar g}\,|\bar\nabla \bar h|_{\bar g}=o(1).
\end{align}
\end{lem} 
\begin{proof}
	\eqref{h estimate} follows from Lemma \ref{blowdown at infinity are planes} using that a flat plane is totally geodesic with respect to $\bar g$. \eqref{H estimate} follows from \eqref{h estimate} and Lemma \ref{geometric expansions}, using that $H=0$ on $\{x\in\mathbb{R}^n:|x|_{\bar g}>1\}\cap \Sigma$.
\end{proof}
To proceed, we recall the monotonicity formula from \cite{simonlectures}.
\begin{lem}[{\cite[(17.4)]{simonlectures}}]
	Let $x_0\in\mathbb{R}^n$ and $0<s<t$. There holds
	\begin{equation} \label{monotonicity formula} 
		\begin{aligned} 
			t^{1-n}\,| B^n_{t}(x_0)\cap  \Sigma|_{\bar g}&=s^{1-n}\,| B^n_{s}(x_0)\cap  \Sigma|_{\bar g}+\int_{(B^n_{t}(x_0)\setminus B^n_{s}(x_0))\cap  \Sigma}|x-x_0|_{\bar g}^{-1-n}\,\bar g(x-x_0,\bar\nu)^2\,\mathrm{d}\bar\mu\\ &\qquad -\frac{1}{n-1}\,\int_{(B^n_{t}(x_0)\setminus B^n_{s}(x_0))\cap  \Sigma}(t^{1-n}-|x-x_0|_{\bar g}^{1-n})\,\bar g(x-x_0,\bar\nu)\,\bar H\,\mathrm{d}\bar\mu\\ &\qquad  
			-\frac{1}{n-1}\,\int_{B^n_{s}(x_0)\cap  \Sigma}(t^{1-n}-s^{1-n})\,\bar g(x-x_0,\bar\nu)\,\bar H\,\mathrm{d}\bar\mu.
		\end{aligned} 
	\end{equation}
\end{lem} 
\begin{proof}
This is 	\cite[17.4]{simonlectures} up to writing the second integral in \cite[17.4]{simonlectures} as the sum of the second and third integral on the right side of \eqref{monotonicity formula}.
\end{proof}
\begin{lem} \label{area estimate sigma bar}
	There holds
	$$
	\sup_{x\in\mathbb{R}^n} \sup_{r>0} \frac{|  B^n_{r}(x)\cap \Sigma|_{\bar g}}{\omega_{n-1}\,r^{n-1}}<\infty.
	$$
\end{lem} 
\begin{proof}
Suppose, for a contradiction, that there are numbers $r_1,\,r_2,\ldots>0$ and points  $x_1,\,x_2,\ldots\in \mathbb{R}^n$ with
\begin{align} \label{contradiction area estimate}  
	\lim_{k\to\infty} \frac{|  B^n_{r_k}(x_k)\cap \Sigma|_{\bar g}}{\omega_{n-1}\,r_k^{n-1}}=\infty.
\end{align} 
Passing to a subsequence and using that  $\Sigma$ is properly embedded, we may assume that either
$$
\lim_{k\to\infty} |x_k|_{\bar g}=\infty\qquad\text{or}\qquad \lim_{k\to\infty} r_k=\infty.
$$
\indent Note that
\begin{align} \label{r vs x} 
	\liminf_{k\to\infty} \frac{|x_k|_{\bar g}}{r_k}\geq 3.
\end{align} 
Indeed,	if not, then $B^n_{r_k}(x_k)\subset B^n_{4\,r_k}(0)$ for a subsequence. This is not compatible with \eqref{contradiction area estimate} and \eqref{area bound at infinity}. \\ \indent 
Let $t_k=|x_k|_{\bar g}/2$.
By \eqref{H estimate}, we have \begin{align} \label{contr mean est}|x_k|_{\bar g}\,\bar H=O(|x_k|_{\bar g}^{-\tau})
	\end{align}  on $B^n_{t_k}(x_k)\cap \Sigma$. We choose $s_k>0$ with $r_k\leq s_k\leq t_k$ such that
\begin{align} \label{sup}  
\frac{| B^n_{s_k}(x_k)\cap \Sigma|_{\bar g}}{\omega_{n-1}\,s_k^{n-1}}= \sup_{r_k\leq r\leq t_k} \frac{| B^n_{r}(x_k)\cap \Sigma|_{\bar g}}{\omega_{n-1}\,r^{n-1}}.
\end{align} 
 Using the monotonicity formula \eqref{monotonicity formula} and \eqref{contr mean est}, we have
\begin{equation} \label{est}  
\begin{aligned} 
\frac{| B^n_{t_k}(x_k)\cap \Sigma|_{\bar g}}{\omega_{n-1}\, t_k^{n-1}}&\geq \frac{| B^n_{s_k}(x_k)\cap \Sigma|_{\bar g}}{\omega_{n-1}\,s_k^{n-1}} -O(|x_k|_{\bar g}^{-1-\tau})\,\int_{(B^n_{t_k}(x_k)\setminus B^n_{s_k}(x_k))\cap \Sigma}\,|x-x_k|_{\bar g}^{2-n}\,\mathrm{d}\bar\mu\\&\qquad -O(|x_k|_{\bar g}^{-\tau})\,\frac{| B^n_{s_k}(x_k)\cap \Sigma|_{\bar g}}{\omega_{n-1}\,s_k^{n-1}}. 
\end{aligned}
\end{equation} 
Using Lemma \ref{layer cake} and \eqref{sup}, we have
$$
\int_{(B^n_{t_k}(x_k)\setminus B^n_{s_k}(x_k))\cap \Sigma}\,|x-x_k|_{\bar g}^{2-n}\,\mathrm{d}\bar\mu=O(|x_k|_{\bar g})\,\frac{| B^n_{s_k}(x_k)\cap \Sigma|_{\bar g}}{\omega_{n-1}\,s_k^{n-1}}.
$$ 
In conjunction with \eqref{contradiction area estimate}, \eqref{sup}, and \eqref{est}, we conclude that
$$
\lim_{k\to\infty}\frac{| B^n_{t_k}(x_k)\cap \Sigma|_{\bar g}}{\omega_{n-1}\,t_k^{n-1}}=\infty.
$$
This is not compatible with \eqref{r vs x}.
\end{proof}

Next, we prove refined estimates on the area growth of $ \Sigma$.

\begin{lem} \label{upper area bound}
	As $s\to\infty$,
	$$
	\frac{| B^n_{s}(0)\cap \Sigma|_{\bar g}}{\omega_{n-1}\,s^{n-1}}\leq 1+O(s^{-\tau}).
	$$
\end{lem}
\begin{proof}

	Using \eqref{H estimate} and the monotonicity formula \eqref{monotonicity formula},  we have, 	for every $0<s<t$,
	\begin{align*} 
		\frac{|B^n_{s}(0)\cap \Sigma|_{\bar g}}{\omega_{n-1}\,s^{n-1}}&\leq\frac{|B^n_{t}(0)\cap \Sigma|_{\bar g}}{{\omega_{n-1}\,t^{n-1}}}+O(1)\,\int_{(B^n_t(0)\setminus B^n_{s}(0))\cap \Sigma} |x|_{\bar g}^{1-n-\tau}\,\mathrm{d}\bar\mu\\
		 &\qquad +O(s^{1-n})\,\int_{  B^n_{s}(0)\cap \Sigma}|x|_{\bar g}^{-\tau}\,\mathrm{d}\bar\mu.
	\end{align*} 
 Using Lemma \ref{layer cake}, \eqref{area bound at infinity}, and \eqref{inner radius}, it follows that 
	$$
	\frac{| B^n_{s}(0)\cap \Sigma|_{\bar g}}{\omega_{n-1}\,s^{n-1}}\leq\frac{|B^n_{t}(0)\cap \Sigma|_{\bar g}}{{\omega_{n-1}\,t^{n-1}}}+O(t^{-\tau})+O(s^{-\tau})+O(s^{1-n}).
	$$
	Letting $t\to\infty$,  the assertion follows, using  \eqref{area bound at infinity} and that $1+\tau<n-1$.
\end{proof}
\begin{lem} \label{lemma lower area bound} 
As $t\to\infty$,
	$$
\frac{| B^n_{t}(0)\cap \Sigma|_{\bar g}}{{\omega_{n-1}\,t^{n-1}}}\geq 1-O(t^{-\tau/(n-1)}).
	$$
\end{lem}
\begin{proof}
		For $t>1$ large, we choose $x_t\in  \Sigma$ with $|x_t|_{\bar g}=t^{(n-1-\tau)/(n-1)}$. We apply the monotonicity formula \eqref{monotonicity formula} with $x_0=x_t$. Letting $s\to 0$, using that $\Sigma$ is properly embedded and \eqref{H estimate}, we obtain 
	$$
	\frac{| B^n_{t}(x_t)\cap \Sigma|_{\bar g}}{\omega_{n-1}\,t^{n-1}}\geq 1  +O(1)\,\int_{   B^n_{t}(x_t)\cap \Sigma}|x-x_t|_{\bar g}^{2-n}\,|x|_{\bar g}^{-1-\tau}\,\mathrm{d}\bar\mu.
	$$
Clearly, $|x|_{\bar g}> |x_t|_{\bar g}/2$ for all $x\in B^n_{|x_t|_{\bar g}/2}(x_t)$. Using Lemma \ref{layer cake} and Lemma \ref{area estimate sigma bar},	we obtain
	$$
	\int_{ B^n_{|x_t|_{\bar g}/2}(x_t)\cap \Sigma}|x-x_t|_{\bar g}^{2-n}\,|x|_{\bar g}^{-1-\tau}\,\mathrm{d}\bar\mu=O(|x_t|_{\bar g}^{-\tau}).
	$$
	Likewise,  $|x-x_t|_{\bar g}\geq|x_t|_{\bar g}/2$ for all $x\in\mathbb{R}^n\setminus B^n_{|x_t|_{\bar g}/2}(x_t)$. It follows that \begin{align*} 
		\int_{  ( B^n_{t}(x_t)\setminus B^n_{|x_t|_{\bar g}/2}(x_t))\cap \Sigma }|x-x_t|_{\bar g}^{2-n}\,|x|_{\bar g}^{-1-\tau}\,\mathrm{d}\bar\mu=O(|x_t|_{\bar g}^{2-n})\,\int_{ B^n_{2\,t}(0)\cap \Sigma}|x|_{\bar g}^{-1-\tau}\,\mathrm{d}\bar \mu.
	\end{align*} 
	Using Lemma \ref{layer cake} again, \eqref{area bound at infinity}, \eqref{inner radius}, and that $1+\tau<n-1$,  we find
	$$
	\int_{ B^n_{2\,t}(0)\cap \Sigma}|x|_{\bar g}^{-1-\tau}\,\mathrm{d}\bar\mu=O(t^{n-2-\tau}).
	$$
Since $0<\tau< n-2$, 
	 $$
-	\tau\,\frac{n-1-\tau}{n-1}< -\frac{\tau}{n-1}= (2-n)\,\frac{n-1-\tau}{n-1}+(n-2-\tau).
	$$
	We conclude that 
	\begin{align} \label{one}  
	\frac{|  B^n_{t}(x_t)\cap \Sigma|_{\bar g}}{\omega_{n-1}\,t^{n-1}}\geq 1 - O(t^{-\tau/(n-1)}).
	\end{align} 
Note that
	\begin{align} \label{two}  
	| B^n_{t}(x_t)\cap \Sigma|_{\bar g}\leq | B^n_{t\,(1+t^{-\tau/(n-1)})}(0)\cap \Sigma|_{\bar g}.
	\end{align} 
	Using  that $1-(1+t^{-\tau/(n-1)})^{1-n}\leq (n-1)\,t^{-\tau/(n-1)}$, we obtain that 
	\begin{equation} \label{three}
		\begin{aligned}  
	\frac{| B^n_{t\,(1+t^{-\tau/(n-1)})}(0)\cap \Sigma|_{\bar g}}{\omega_{n-1}\,t^{n-1}}&\leq \frac{|  B^n_{t\,(1+t^{-\tau/(n-1)})}(0)\cap \Sigma|_{\bar g}}{\omega_{n-1}\,t^{n-1}\,(1+t^{-\tau/(n-1)})^{n-1}}\\&\qquad +O(t^{-\tau/(n-1)})\,\frac{|  B^n_{t\,(1+t^{-\tau/(n-1)})}(0)\cap \Sigma|_{\bar g}}{t^{n-1}}.
	\end{aligned} 
	\end{equation} 
	By \eqref{area bound at infinity},
	\begin{align} \label{four}
	\frac{|  B^n_{t\,(1+t^{-\tau/(n-1)})}(0)\cap \Sigma|_{\bar g}}{t^{n-1}}=O(1).
	\end{align} 
	Let $\tilde t=t\,(1+t^{-\tau/(n-1)})$. Combing \eqref{one}, \eqref{two}, \eqref{three}, and \eqref{four}, we see that, as $\tilde t\to\infty$,  
	$$
\frac{| B^n_{\tilde t}(0)\cap \Sigma|_{\bar g}}{{\omega_{n-1}\,\tilde t^{n-1}}}\geq 1-O(\tilde t^{-\tau/(n-1)}).
$$
The assertion follows.
\end{proof}
\begin{lem} \label{excess estimate}
As $s\to\infty$,
	$$
	\int_{ \Sigma\setminus B^n_{s}(0)}|x|_{\bar g}^{-1-n}\,\bar g(x,\bar\nu)^2\,\mathrm{d}\bar\mu=O(s^{-\tau/(n-1)}).
	$$ 
\end{lem}
\begin{proof}
	We apply the monotonicity formula \eqref{monotonicity formula} with $x_0=0$. In conjunction with Lemma \ref{upper area bound} and Lemma \ref{lemma lower area bound}, letting $t\to\infty$, we obtain
	\begin{align*}
		\int_{ \Sigma\setminus B^n_{s}(0)}|x|_{\bar g}^{-1-n}\,\bar g(x,\bar\nu)^2\,\mathrm{d}\bar\mu&=O(1)\,\int_{ \Sigma\setminus B^n_{s}(0)}|x|_{\bar g}^{1-n}\,\bar g(x,\bar\nu)\,\bar H\,\mathrm{d}\bar\mu\\ &\qquad  
		+O(s^{1-n})\,\int_{ B^n_{s}(0)\cap \Sigma}\bar g(x,\bar\nu)\,\bar H\,\mathrm{d}\bar\mu+O(s^{-\tau/(n-1)}).
	\end{align*}
	By Lemma \ref{layer cake}, \eqref{area bound at infinity}, \eqref{inner radius}, and \eqref{H estimate}, 
	$$
	\int_{ \Sigma\setminus B^n_{s}(0)}|x|_{\bar g}^{1-n}\,\bar g(x,\bar\nu)\,\bar H\,\mathrm{d}\bar\mu=O(s^{-\tau})
	\qquad \text{and}\qquad
\int_{  B^n_{s}(0)\cap \Sigma}\bar g(x,\bar\nu)\,\bar H\,\mathrm{d}\bar\mu=O(s^{n-1-\tau}).
	$$
	\indent The assertion follows from these estimates.
\end{proof}
Next, we  show that there is only one  tangent plane at infinity that can arise in the setting of Lemma \ref{blowdown at infinity are planes}. To this end, we apply an argument developed by B.~White \cite[Theorem 3]{whitetangentconeuniqueness} to study the uniqueness of tangent planes at isolated singularities of area-minimizing  surfaces. This argument has been adapted to study the uniqueness of tangent planes at infinity of certain minimal surfaces in $\mathbb{R}^3$ by P.~Gallagher \cite[Lemma 2.4]{gallgheruniqueness}.
\begin{lem} \label{area estimate of normal}  Let $F:\Sigma\setminus B^n_1(0)\to S^{n-1}_1(0)$ be given by $$F(x)=\frac{x}{|x|_{\bar g}}.$$ As $s\to\infty$, 
	$$
	|F(\Sigma\setminus B^n_s(0))|_{\bar g}=O(s^{-\tau/(2\,n-2)}).
	$$ 
\end{lem} 
\begin{proof}
By the area formula,
	$$
	|F(\Sigma\setminus B^n_s(0))|_{\bar g}=\int_{\Sigma\setminus B^n_s(0)}|x|_{\bar g}^{-n}\,|\bar g(x,\bar\nu)|\,\mathrm{d}\bar\mu.
	$$
	By the Cauchy-Schwarz inequality, 
	\begin{align*} 
&\left(\int_{\Sigma\setminus B^n_s(0)}|x|_{\bar g}^{-n}\,|\bar g(x,\bar\nu)|\,\mathrm{d}\bar\mu\right)^2	\\&\qquad \leq \sum_{k=0}^\infty\int_{ (B^n_{2^{k+1}\,s}(0)\setminus B^n_{2^{k}s}(0))\cap \Sigma}|x|_{\bar g}^{-1-n}\,\bar g(x,\bar\nu)^2\,\mathrm{d}\bar\mu\,\int_{ (B^n_{2^{k+1}\,s}(0)\setminus B^n_{2^{k}s}(0))\cap \Sigma}|x|_{\bar g}^{1-n}\,\mathrm{d}\bar\mu. 
	\end{align*} 
Invoking Lemma \ref{upper area bound} and Lemma \ref{excess estimate}, we obtain
$$
\left(\int_{\Sigma\setminus B^n_s(0)}|x|_{\bar g}^{-n}\,|\bar g(x,\bar\nu)|\,\mathrm{d}\bar\mu\right)^2\leq O(1)\,\sum_{k=1}^\infty (2^k\,s)^{-\tau/(n-1)}=O(s^{-\tau/(n-1)}).
$$
\end{proof}
\begin{coro} \label{hyperplane uniqueness}
The  tangent planes in Lemma \ref{blowdown at infinity are planes} all agree.
\end{coro}
\begin{proof}
	Suppose, for a contradiction, that  $\pi_1, \pi_2\subset\mathbb{R}^n$ are two different tangent planes at infinity that arise as in Lemma \ref{blowdown at infinity are planes}. Let $s>1$ be sufficiently large such that $\lambda^{-1}\,\Sigma$ and $S^{n-1}_1(0)$ intersect transversely for every $\lambda>s$ and let $\lambda_2>\lambda_1>s$ be such that $S^{n-1}_{1}(0)\cap\lambda_1^{-1}\,\Sigma$ and $S^{n-1}_{1}(0)\cap\lambda_2^{-1}\,\Sigma$ are close to $S^{n-1}_1(0)\cap \pi_1$ and $S^{n-1}_1(0)\cap \pi_2$, respectively.  Note that $\{S^{n-1}_1(0)\cap \lambda^{-1}\,\Sigma:\lambda\in[\lambda_1,\lambda_2]\}$ is a homotopy of $S^{n-1}_1(0)\cap \lambda_1^{-1}\,\Sigma$ and $S^{n-1}_1(0)\cap \lambda_2^{-1}\,\Sigma$. Since a closed hypersurface of $\mathbb{R}^{n-1}$ that encloses the origin is not contractible in $\mathbb{R}^{n-1}\setminus\{0\}$, it follows that   $F((B^n_{\lambda_2}(0)\setminus B^n_{\lambda_1}(0))\cap\Sigma)$ contains at least two of the four components of 
	$S^{n-1}_1(0)\setminus (\lambda_1^{-1}\Sigma\cup \lambda_2^{-1}\,\Sigma)$. Using this and that the area of each of these four components is close to the area of one of the four components of $S^{n-1}_1(0)\setminus(\pi_1\cup \pi_2)$, it follows that 
	$$
	\liminf_{s\to\infty} |F(\Sigma\setminus B^n_{s}(0))|_{\bar g}>0.
	$$
	This is not compatible with Lemma \ref{area estimate of normal}.
\end{proof}

\begin{proof}[{Proof of Proposition \ref{graph estimate}}]
	Using Lemma \ref{components reduction}, we may assume that \eqref{area bound at infinity} holds.
Using Lemma  \ref{blowdown at infinity are planes} and Corollary \ref{hyperplane uniqueness}, we see that, after a rotation, there are $r_0>1$ and $\psi\in C^\infty(\mathbb{R}^{n-1})$ with
	$$
	\Sigma\setminus B^n_{r_0}(0)\subset \{(y,\psi(y)):y\in\mathbb{R}^{n-1}\}
	$$
and
	\begin{align} \label{preliminary graph estimate} 
		|y|_{\bar g}^{-1}\,|\psi|+|\bar\nabla \psi|_{\bar g}+|y|_{\bar g}\,|\bar\nabla^2 \psi|_{\bar g}=o(1).
	\end{align} 
\indent 
	Let $\upsilon\in C^\infty( \Sigma)$ be given by $\upsilon=\bar g(x,\bar\nu)$. Note that $\bar \nabla \upsilon=x^\top\lrcorner \bar h$ where $\lrcorner$ is the insertion operator. Consequently, $$\bar \Delta_\Sigma \upsilon=H-|\bar h|^2_{\bar g}\,\upsilon+\operatorname{div}_\Sigma \bar h(x^\top).$$ Using also the Codazzi equation $\operatorname{div}_\Sigma \bar h(x^\top)=\bar \nabla_{x^\top} H$ and \eqref{H estimate}, we obtain
	\begin{align} \label{v PDE} 
	\bar\Delta_\Sigma \upsilon +|\bar h|^2_{\bar g}\,\upsilon=O(|x|_{\bar g}^{-1-\tau}).
	\end{align} 
	Let $x\in \Sigma$ and $r=|x|_{\bar g}/4$. By \eqref{preliminary graph estimate}, $\Sigma_{x,r}=r^{-1}\,(-x+B^n_{3\,r}(x)\cap \Sigma)$ is smoothly close to  $B^{n-1}_3(0)$  provided that $|x|_{\bar g}$ is sufficiently large. Let $\upsilon_{x,r}\in C^\infty(\Sigma_{x,r})$ be given by $\upsilon_{x,r}(z)=v(x+r\,z)$. By \eqref{h estimate} and \eqref{v PDE}, 
	$$
	\bar \Delta_{\Sigma_{x,r}} \upsilon_{x,r}=O(1)\,\upsilon_{x,r}+O(r^{1-\tau}).
	$$
	 Using the interior $L^2$-estimate \cite[Theorem 9.11]{gilbargtrudinger} and the Sobolev embedding theorem, recalling that $3\leq n\leq 7$, we have
	 $$
	\left(\int_{ B^n_{2}(0)\cap \Sigma_{x,r}}\upsilon_{x,r}^6\,\mathrm{d}\bar\mu\right)^\frac16\leq O(1) \left(\int_{  \Sigma_{x,r}}\upsilon_{x,r}^2\,\mathrm{d}\bar\mu\right)^\frac12 +O(r^{1-\tau})
	 $$
	 so that 
	$$
r^{-1}\,\left(r^{1-n}\,\int_{ B^n_{2\,r}(x)\cap \Sigma}\upsilon^6\,\mathrm{d}\bar\mu\right)^\frac16\leq O(r^{-1})\,\left(r^{1-n}\,\int_{ B^n_{3\,r}(x)\cap \Sigma}\upsilon^2\,\mathrm{d}\bar\mu\right)^\frac12 +O(r^{-\tau}).
	$$
	By Lemma \ref{excess estimate},
	$$
	r^{1-n}\,\int_{  B^n_{3\,r}(x)\cap \Sigma}\upsilon^2\,\mathrm{d}\bar\mu=O(r^{-\tau/(n-1)}).
	$$
	It follows that
	$$
r^{-1}\,\left(r^{1-n}\,\int_{  B^n_{2\,r}(x)\cap \Sigma}\upsilon^6\,\mathrm{d}\bar\mu\right)^\frac16=O(r^{-\tau/(2\,n-2)}).
	$$
Using the interior $L^6$-estimate \cite[Theorem 9.11]{gilbargtrudinger} and the Sobolev embedding theorem, we have
		$$
r^{-1}\	\left(r^{1-n}\,\int_{  B^n_{r}(x)\cap \Sigma}\upsilon^{7}\,\mathrm{d}\bar\mu\right)^\frac1{7}=O(r^{-\tau/(2\,n-2)}).
	$$
	Finally, using the interior $L^7$-estimate \cite[Theorem 9.11]{gilbargtrudinger} and the Sobolev embedding theorem, we 
	 conclude that
	\begin{align} \label{v estimate} 
|x|_{\bar g}^{-1}\,|\upsilon|	+|\bar\nabla \upsilon|_{\bar g}=O(|x|_{\bar g}^{-\tau/(2\,n-2)}).
	\end{align}
Let $\tilde \upsilon:\{y\in \mathbb{R}^{n-1}:|y|_{\bar g}>r_0\}\to\mathbb{R}$ be given by $$\tilde \upsilon(y)=\frac{-\bar g((\bar\nabla \psi)(y),y)+\psi(y)}{\sqrt{1+|(\bar\nabla \psi)(y)|^2_{\bar g}}}.$$
Note that $\tilde \upsilon(y)=\upsilon((y,\psi(y)))$ and
	$$
	\bar\nabla \tilde \upsilon=-(1+|\bar\nabla \psi|^2_{\bar g})^{-1/2}\,y\lrcorner \bar\nabla^2\psi-(1+|\bar\nabla \psi|_{\bar g}^2)^{-1}\,\tilde \upsilon\,\bar\nabla \psi\lrcorner\bar\nabla^2\psi .
	$$
In conjunction with \eqref{preliminary graph estimate} and \eqref{v estimate}, 	 we obtain the improved estimate
	\begin{align} \label{hoelder decay}
	y\lrcorner\bar\nabla^2\psi=O(|y|_{\bar g}^{-\tau/(2\,n-2)}).
	\end{align} 
\indent Let $\alpha>0$ and suppose that $	y\lrcorner\bar\nabla^2\psi=O(|y|_{\bar g}^{-\alpha})$ so that $\bar\nabla \psi=O(|y|_{\bar g}^{-\alpha})$. Using \eqref{H estimate}, \eqref{preliminary graph estimate}, and the minimal surface equation
$$
\bar\Delta_{\mathbb{R}^{n-1}} \psi=(1+|\bar\nabla \psi|^2)^{-1}\,\bar \nabla^2 \psi(\bar \nabla \psi,\bar \nabla \psi)+\sqrt{1+|\bar\nabla \psi|^2}\,\bar H,
$$
we obtain that 
$$
|y|_{\bar g}\,\bar\Delta_{\mathbb{R}^{n-1}} \psi=O(|y|_{\bar g}\,|\bar\nabla \psi|_{\bar g}^2\,|\bar\nabla^2 \psi|_{\bar g})+O(|y|_{\bar g}^{-\tau})=O(|y|_{\bar g}^{\max\{-2\,\alpha,-\tau\}}).
$$
Proceeding as in \cite[p.~16]{Carlottocalcvar}, we see that, given $\varepsilon>0$, there are $\tilde \psi,\,\hat\psi\,\in C^\infty(\Sigma)$ such that $\psi=\tilde\psi+\hat\psi$ where $\hat\psi$ satisfies
$$
|y|_{\bar g}^{-1}\,|\hat\psi|+|\bar\nabla \hat\psi|_{\bar g}+|y|_{\bar g}\,|\bar \nabla ^2 \hat\psi|_{\bar g}=O(|y|_{\bar g}^{\max\{-2\,\alpha,-\tau\}+\varepsilon}).
$$ 
and
 $\tilde\psi$ is harmonic with $$ |y|_{\bar g}^{-1}\,|\tilde\psi-a|+|\bar\nabla \tilde\psi|_{\bar g}+|y|_{\bar g}\,|\bar\nabla^2 \tilde\psi|_{\bar g}+=O(|y|_{\bar g}^{2-n})$$ for some $a\in\mathbb{R}$ if $3<n\leq 7$ and 
 $$
 |y|^{-1}_{\bar g}\,|\tilde\psi-a\,\log|y|_{\bar g}|+|\bar\nabla \tilde\psi|_{\bar g}+|y|_{\bar g}\,|\bar\nabla^2 \tilde\psi|_{\bar g}=O(|y|_{\bar g}^{-1})
 $$
for some $a\in \mathbb{R}$ if $n=3$, respectively.
Iterating this argument, we obtain \eqref{graph estimate optimal 2} from \eqref{hoelder decay}.
\end{proof}
\begin{coro} \label{l2 second fundamental form}
	Suppose that $n-3<\tau< n-2$. There holds
	$$
	\int_{\Sigma} |h|^2\,\mathrm{d}\mu<\infty\qquad\text{and}\qquad  \int_{\Sigma} |{Ric}(\nu,\nu)|\,\mathrm{d}\mu<\infty.
	$$
\end{coro}
\begin{proof}
	This follows from Proposition \ref{graph estimate}, Lemma \ref{geometric expansions}, and \eqref{AF section 2}.
\end{proof}
For the next lemma, recall the definition \eqref{def adm mass} of the mass of an asymptotically flat manifold.
\begin{lem} \label{vanishing mass} 
Suppose that $3<n\leq 7$ and  $n-3<\tau< n-2$.	Each end of the  Riemannian $(n-1)$-manifold $(\Sigma,g|_{\Sigma})$ is asymptotically flat with mass zero. 
\end{lem}
\begin{proof}
Fix $i \in \{1, \ldots, N\}$ and let $\varphi : \mathbb{R}^{n-1} \to \mathbb{R}^n$ be the chart given by $\varphi(y) = (y, \psi_i(y))$ where $\psi_i\in C^\infty(\mathbb{R}^{n-1})$ is as in Proposition \ref{graph estimate}. Note that, for every $\varepsilon>0$,
\[
|(\varphi^* g)-\bar g|_{\mathbb{R}^{n-1}}|_{\bar g|_{\mathbb{R}^{n-1}}}+|y|_{\bar g}\,|\bar D (\varphi^* g)|_{\bar g|_{\mathbb{R}^{n-1}}}+|y|^2_{\bar g}\,|\bar D^2 (\varphi^* g)|_{\bar g|_{\mathbb{R}^{n-1}}}
= O (|y|_{\bar g}^{-\tau+\varepsilon}).
\]
It follows that
$$
\lambda^{-1}\,\int_{S^{n-2}_{\lambda}(0)}\sum_{i,\,j=1}^{n-1}x^i\,\left[(\bar D_{e_j}(\varphi^*g))(e_i,e_j)-(\bar D_{e_i}(\varphi^*g))(e_j,e_j)\right]\,\mathrm{d}\bar\mu=O(\lambda^{n-3-\tau+\varepsilon}).
$$
Using that $\tau > n - 3$, the assertion follows.
\end{proof}

\section{Stability with respect to asymptotically constant variations}
In this section, we assume that $(M,g)$ is  a Riemannian manifold of dimension  $3< n\leq 7$ which is  asymptotically flat of rate $\tau$ where
\begin{align}
	\label{tau condition} n-3<\tau< n-2.
\end{align}
 \indent Let $\Omega_1,\,\Omega_2,\hdots\subset M$ be isoperimetric regions in $(M,g)$  with $| \Omega_k|\to\infty$ such that $\Omega_k$ converges locally smoothly to a region $\Omega\subset M$ whose boundary $\Sigma=\partial \Omega$ is noncompact and area-minimizing. \\ \indent 
 We define the area radius of $\Sigma_k=\partial \Omega_k$ to be $$\lambda(\Sigma_k)=\left(\frac{|\Sigma_k|}{n\,\omega_{n}}\right)^{\frac{1}{n-1}}.$$
 Recall from Lemma \ref{lem:locallyisoperimetricboundary}  that $\Sigma$ has one noncompact component and that the other components of $\Sigma$ are contained in $\partial M$. By Proposition  \ref{smooth local convergence}, passing to a subsequence, there holds $\lambda(\Sigma_k)^{-1}\,(\Omega_k\setminus B_1)\to B^n_1(\xi)$ locally smoothly in $\mathbb{R}^n\setminus \{0\}$ for some $\xi\in\mathbb{R}^n$ with $|\xi|_{\bar g}=1$. By Proposition \ref{graph estimate}, $\Sigma$ is  asymptotic to a coordinate hyperplane. After a rotation, we may assume that the normal of this plane pointing towards $\Omega$ is the standard unit vector $e_n$.
  	\begin{figure}\centering
 	\includegraphics[width=0.7\linewidth]{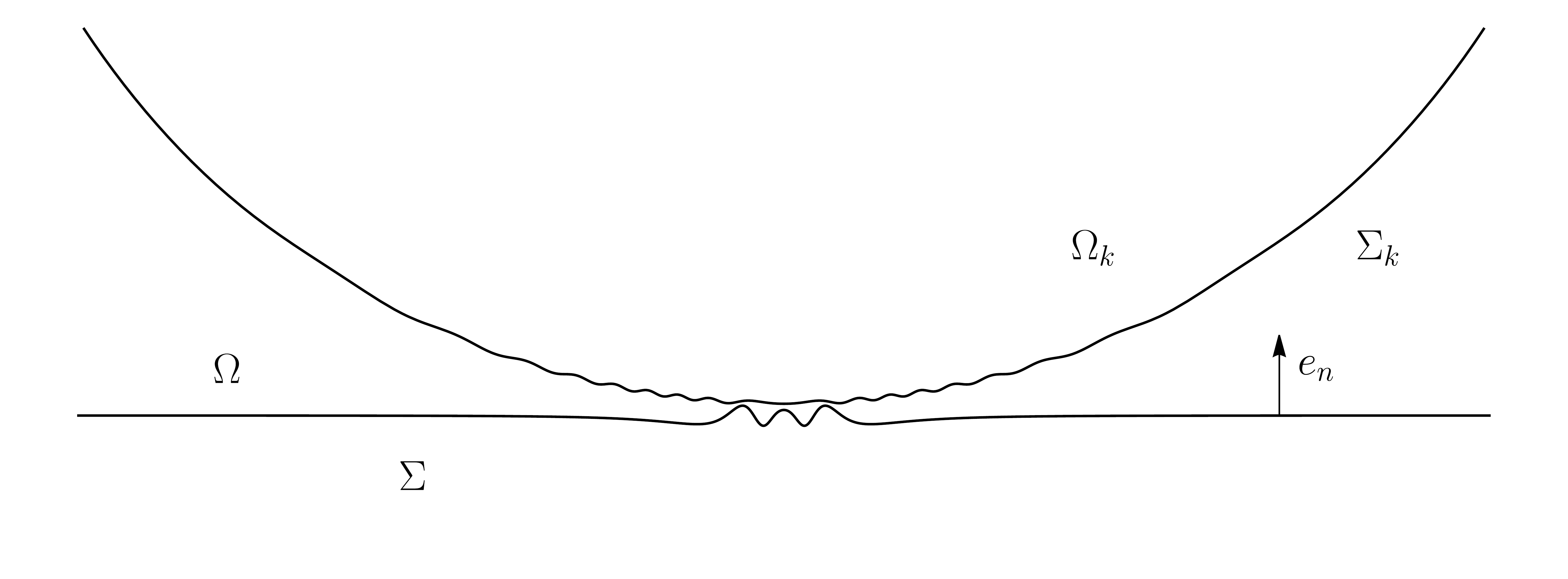}
 	\caption{The hypersurfaces $\Sigma_k=\partial \Omega_k$ converge locally uniformly to $\Sigma=\partial \Omega$.
 	}
 	\label{boundary}
 \end{figure}
\\  \indent  
The goal of this section is to show that $\Sigma$ is  stable with respect to asymptotically constant variations, i.e., that
$$
\int_{\Sigma} (|h|^2+Ric(\nu,\nu))\,(1+f)^2\,\mathrm{d}\mu\leq\int_{\Sigma}|\nabla f|^2\,\mathrm{d}\mu
$$
for all $f\in C^\infty_c(\Sigma)$. 
 To this end, we will study the second variation of area of $\Sigma_k=\partial\Omega_k$ with respect to a Euclidean translation; see Figure \ref{boundary}. \\ \indent  
 Let $\chi\in C^\infty(M,TM)$ be a vector field with $\chi=e_n$ in $M\setminus B_2$ and $\chi=0$ in $B_1$.  Let $u_k,\,v_k\in C^\infty(\Sigma_k)$ be the functions
$$
u_k=g(\chi,\nu)\qquad \text{and}\qquad  v_k=-h(\chi^\top,\chi^\top).
$$ 
Note that $u_k$ and $v_k$ are, respectively, the initial normal speed and initial normal acceleration of the variation $\{x+s\,\chi(s):x\in \Sigma_k\}_{|s|<\varepsilon}$ where $\varepsilon>0$ is chosen sufficiently small.  
\begin{lem} \label{replacement surfaes}
There are closed hypersurfaces $\hat{\Sigma}_1,\,\hat{\Sigma}_2,\ldots\subset \mathbb{R}^n$ with 
	\begin{align}
\label{tilde one}		&\circ \qquad  B^n_1(0)\cap \hat \Sigma_k=\emptyset,\qquad\qquad\qquad\qquad\qquad\qquad\qquad\qquad\qquad\\
\label{tilde two}		&\circ \qquad \hat \Sigma_k\setminus B^n_2(0)=\Sigma_k\setminus B_2, \\
\label{tilde three}		&\circ \qquad | B^n_2(0)\cap \hat \Sigma_k|_{\bar g}=O(1), \text{ and} \\
\label{tilde four}		&\circ \qquad  h(\hat\Sigma_k)=O(1).
	\end{align}
\end{lem}
\begin{proof}
		Since $\Sigma_k$ converges to $\Sigma$ locally smoothly,  there exist closed hypersurfaces $\tilde \Sigma_1,\,\tilde \Sigma_2,\,\ldots\subset M$ such that 
 		\begin{align*}
		&\circ \qquad  \tilde \Sigma_k\setminus B_{7/4}=\Sigma_k\setminus B_{7/4},\qquad\qquad\qquad\qquad\qquad\qquad\qquad\qquad\qquad\\
		&\circ \qquad B_{3/2}\cap \tilde \Sigma_k=B_{3/2}\cap \Sigma ,\text{ and}\\
		&\circ \qquad \tilde \Sigma_k\text{ converges  to $\Sigma$ locally smoothly}.
 	\end{align*}
Let $r\in (1,3/2)$ be such that  $\Sigma$ and $\partial B_{r}$ intersect transversely and $\tilde \Sigma \subset M$ be  a hypersurface such that   $B_{r/2+1/2}\cap \tilde \Sigma=\emptyset$ and $\tilde \Sigma\setminus B_r=\Sigma\setminus B_r$. The concatenate hypersurfaces $\hat\Sigma_1,\,\hat \Sigma_2,\ldots\subset \mathbb{R}^n$ given by $\hat \Sigma_k=(\tilde \Sigma_k\setminus B_r)\cup (\tilde \Sigma \cap  B_r) $ have all of the asserted properties.
\end{proof}
Recall from \eqref{H est} that $\lambda(\Sigma_k)\,H(\Sigma_k)=(n-1)+o(1)$.
\begin{lem}
	As $k\to\infty$, \label{first volume var}
	$$
\lambda(\Sigma_k)^{1-n}\,	\int_{\Sigma_k}u_k\,\mathrm{d}\mu=O(\lambda(\Sigma_k)^{-\tau}).
	$$
\end{lem}
\begin{proof}
Let $\hat{\Sigma}_1,\,\hat{\Sigma}_2,\ldots\subset \mathbb{R}^n$ be the hypersurfaces from Lemma \ref{replacement surfaes}.  
Note that
	$$
	\int_{\Sigma_k} g(\chi,\nu)\,\mathrm{d}\mu =\int_{\hat \Sigma_k} g(\chi,\nu)\,\mathrm{d}\mu- \int_{B_2(0)\cap \hat  \Sigma_k} g(\chi,\nu)\,\mathrm{d}\mu+\int_{ B_2\cap \Sigma_k} g(\chi,\nu)\,\mathrm{d}\mu
	$$ 
By \eqref{tilde three}, 
	$$
	 \int_{ B_2\cap \hat  \Sigma_k} g(\chi,\nu)\,\mathrm{d}\mu=O(1).
	$$
	By Lemma \ref{area growth estimate}, $\limsup_{k\to \infty}|B_2\cap \Sigma_k|<\infty$, so that 
	$$
	\int_{ B_2\cap \Sigma_k} g(\chi,\nu)\,\mathrm{d}\mu=O(1).
	$$
Using Lemma \ref{geometric expansions}, we see that
$$
\int_{\hat \Sigma_k} g(\chi,\nu)\,\mathrm{d}\mu=\int_{\hat \Sigma_k} \bar g (e_n,\nu)\,\mathrm{d}\bar \mu+O(1)\,\int_{\hat \Sigma_k} |x|_{\bar g}^{-\tau}\,\mathrm{d}\bar\mu.
$$
By the divergence theorem,
$$
\int_{\hat \Sigma_k} \bar g (e_n,\nu)\,\mathrm{d}\bar \mu=0.
$$ 
By \eqref{tilde two}, 
$$
\int_{\hat \Sigma_k} |x|_{\bar g}^{-\tau}\,\mathrm{d}\bar\mu=\int_{ \Sigma_k\setminus B_2} |x|_{\bar g}^{-\tau}\,\mathrm{d}\bar\mu+\int_{B_2\cap\hat \Sigma_k} |x|_{\bar g}^{-\tau}\,\mathrm{d}\bar\mu.
$$
By \eqref{tilde one} and \eqref{tilde three}, 
$$
\int_{B_2\cap\hat \Sigma_k} |x|_{\bar g}^{-\tau}\,\mathrm{d}\bar\mu=O(1).
$$
Finally, using Lemma \ref{area growth estimate} and that $\tau<n-2$, 
$$
\int_{\Sigma_k\setminus B_2} |x|_{\bar g}^{-\tau}\,\mathrm{d}\bar\mu=O(\lambda(\Sigma_k)^{n-1-\tau})+O(1)=O(\lambda(\Sigma_k)^{n-1-\tau}).
$$
\indent The assertion follows from these estimates. 
\end{proof}
\begin{lem} \label{second volume var} 
	There holds, as $k\to\infty$,
	$$
\lambda(\Sigma_k)^{2-n}\,	\int_{\Sigma_k} v_k+H\,u_k^2\,\mathrm{d}\mu=O(\lambda(\Sigma_k)^{-\tau}).
	$$
\end{lem}
\begin{proof}
Let $\hat{\Sigma}_1,\,\hat{\Sigma}_2,\ldots\subset \mathbb{R}^n$ be the hypersurfaces from Lemma \ref{replacement surfaes}. Note that
\begin{align*} 
\int_{\Sigma_k}  v_k+H\,u_k^2\,\mathrm{d}\mu &=\int_{\hat \Sigma_k}  -h(\chi^\top,\chi^\top)+H\,g(\chi,\nu)^2\,\mathrm{d}\mu
\\&\qquad - \int_{B_2\cap \hat  \Sigma_k}  -h(\chi^\top,\chi^\top)+H\,g(\chi,\nu)^2\,\mathrm{d}\mu+\int_{ B_2\cap \Sigma_k} v_k+H\,u_k^2\,\mathrm{d}\mu.
\end{align*} 
   Using the area estimate \eqref{appendix area est} and the curvature estimates \eqref{H est} and Lemma \ref{lem:locisoconverge}, we see that
	$$
	\int_{B_2\cap \Sigma_k} v_k+H\,u_k^2\,\mathrm{d}\mu=O(1).
	$$
	Similarly, by \eqref{tilde three} and \eqref{tilde four},
	$$
	\int_{B_2\cap \hat \Sigma_k} -h(\chi^\top,\chi^\top)+H\,g(\chi,\nu)^2\,\mathrm{d}\mu=O(1).
	$$
	Using   Lemma \ref{geometric expansions} and the curvature estimates \eqref{H est} and \eqref{hcirc est}, we obtain, on $\tilde \Sigma_k$,
	$$
-h(\chi^\top,\chi^\top)+H\,g(\chi,\nu)^2=-\bar h(e_n^{\bar\top},e_n^{\bar \top})+\bar H\,\bar g(e_n,\bar \nu)^2+O(|x|_{\bar g}^{-1-\tau}).
	$$
By \eqref{tilde two}, 
$$
\int_{\hat \Sigma_k} |x|_{\bar g}^{-1-\tau}\,\mathrm{d}\bar\mu=\int_{ \Sigma_k\setminus B_2} |x|_{\bar g}^{-1-\tau}\,\mathrm{d}\bar\mu+\int_{B_2\cap\hat \Sigma_k} |x|_{\bar g}^{-1-\tau}\,\mathrm{d}\bar\mu.
$$
By \eqref{tilde one} and \eqref{tilde three}, 
$$
\int_{B_2\cap\hat \Sigma_k} |x|_{\bar g}^{-1-\tau}\,\mathrm{d}\bar\mu=O(1).
$$		
By Lemma \ref{area growth estimate}, using that $1+\tau<n-1$, 
$$
\int_{\Sigma_k\setminus B_2} |x|_{\bar g}^{-1-\tau}\,\mathrm{d}\bar\mu=O(\lambda(\Sigma_k)^{n-2-\tau})+O(1)=O(\lambda(\Sigma_k)^{n-2-\tau}).
$$
	Combining these equations,
$$
\int_{\Sigma_k} v_k+H\,u_k^2\,\mathrm{d}\mu=\int_{\hat \Sigma_k}-\bar h(e_n^{\bar\top},e_n^{\bar \top})+\bar H\,\bar g(e_n,\bar \nu)^2\,\mathrm{d}\bar \mu+O(\lambda(\Sigma_k)^{n-2-\tau}).
$$ 
Note that $ \bar{\operatorname{div}}_{\hat \Sigma_k}(\bar g(e_n,\bar \nu)\,e_n^\top)=\bar h(e_n^\top,e_n^\top)-\bar H\,\bar g(e_n,\bar \nu)^2$. Consequently, by the first variation formula,
$$
\int_{\hat \Sigma_k}-\bar h(e_n^{\bar \top},e_n^{\bar \top})+\bar H\,\bar g(e_n,\bar \nu)^2\,\mathrm{d}\bar \mu=0.
$$
\indent The assertion follows from these estimates.
\end{proof}
Fix $\eta\in C^\infty(\mathbb{R})$ with 
\begin{itemize}
	\item[$\circ$]$\eta(t)=0$ if $t\leq 1$, 
	\item[$\circ$]$\eta(t)>0$ if $1<t<2$, and
	\item[$\circ$] $\eta(t)=0$ if $t\geq 2$.
\end{itemize}
 We define
$$
\kappa_k=\lambda(\Sigma_k)^{1-n}\,\int_{\Sigma_k} \eta\left(\lambda(\Sigma_k)^{-1}\,|x|_{\bar g}\right)\,\mathrm{d}\mu
\qquad\text{and}\qquad 
\kappa=\int_{S^{n-1}_1(\xi)}\eta(|x|_{\bar g})\,\mathrm{d}\bar\mu.
$$
Note that $\kappa>0$.
\begin{lem} 
	Passing to a subsequence, there holds
	$
 \kappa_k\to\kappa. 
	$ \label{test integral}
\end{lem}
\begin{proof}
	This follows from Proposition \ref{smooth local convergence}.
\end{proof} 
Let $f\in C_c^\infty(M)$ be a function whose support is disjoint from $\partial M$. We define
\begin{align*} 
\tilde u_k(x)=\begin{dcases}
	\alpha_k\,\lambda(\Sigma_k)^{-\tau}\,\eta\left(\lambda(\Sigma_k)^{-1}\,|x|_{\bar g}\right) \qquad&\text{if } x\notin B_1,\\
	0&\text{else}.
\end{dcases}
\end{align*} 
Here, $\alpha_k\in\mathbb{R}$ is chosen such that
\begin{align} \label{tilde u compatibility}
\int_{\Sigma_k}u_k+f+\tilde u_k\,\mathrm{d}\mu=0.
\end{align} 
Using Lemma \ref{first volume var} and Lemma \ref{test integral}, we see that $\alpha_k=O(1)$. Next, we define

\begin{align*}
\tilde v_k(x)=\begin{dcases} \beta_k\,\lambda(\Sigma_k)^{-1-\tau}\,\eta(\lambda(\Sigma_k)^{-1}\,|x|_{\bar g})\qquad&\text{if }x\notin B_1,
	\\0&\text{else},
	\end{dcases}
\end{align*} 
where $\beta_k\in\mathbb{R}$ is chosen such that
\begin{align} \label{tilde vcompatibility}  
	\int_{\Sigma_k} v_k+\tilde v_k+H\,(u_k+f+\tilde u_k)^2\,\mathrm{d}\mu=0.
\end{align} 
Using Lemma \ref{second volume var}, Lemma \ref{test integral}, and \eqref{H est}, we see that  $\beta_k=O(1)$. \\ \indent Note that there is $\varepsilon>0$ such that, for all $s\in(-\varepsilon,\varepsilon)$, $$\Sigma_k(s)=\{\exp_x\left(U_k(x,s)\,\nu(\Sigma_k)(x)\right):x\in \Sigma_k\}$$
where
$$
 U_k(x,s)=	s\,(u_k+f+\tilde u_k)(x)+\tfrac12\,s^2\,(v_k+\tilde v_k)(x)
$$
is an embedded hypersurface in $M$ that bounds a region  $\Omega_k(s)$.
\begin{lem}
	There holds \label{volume preserving}
	$$
	\frac{d}{ds}\bigg|_{s=0}|\Omega_k(s)|=	\frac{d^2}{ds^2}\bigg|_{s=0}|\Omega_k(s)|=0.
	$$
	\end{lem}  
\begin{proof}
	This follows from Lemma \ref{variation of area and volume} using \eqref{tilde u compatibility} and \eqref{tilde vcompatibility}.
\end{proof}
By Proposition \ref{graph estimate},  $\partial B_r$ and $\Sigma$ intersect transversely for all sufficiently large $r>1$. Increasing $r>1$ if necessary, we may arrange that  \begin{align} 
	\label{spt f}\operatorname{spt}(f)\subset B_r. \end{align}  \indent 
For Proposition \ref{gamma limit lemma} below, let $u\in C^\infty(\Sigma)$ be the function given by $u=g(\chi,\nu)$.
\begin{prop}There holds \label{gamma limit lemma} 
	$$
	\int_{ B_r\cap \Sigma} (|h|^2+{Ric}(\nu,\nu))\,(u+f)^2\,\mathrm{d}\mu\leq \int_{B_r\cap \Sigma}|\nabla (u+f)|^2\,\mathrm{d}\mu+O(r^{n-3-\tau}).
	$$
\end{prop}
\begin{proof}

By Lemma \ref{volume preserving}, the variation $\{\Sigma_k(s)\}_{|s|<\varepsilon}$ is volume-preserving up to second order. Since $\Omega_k$ is isoperimetric, it follows that
	$$
	\frac{d^2}{ds^2}\bigg|_{s=0}| \Sigma_k(s)|\geq 0.
	$$
	By \eqref{second area},
	$$
	\frac{d^2}{ds^2}\bigg|_{s=0}| \Sigma_k(s)|_{g}=\int_{\Sigma_k}H\,\ddot{ U}_k+H^2\,\dot{ U}^2_k+|\nabla \dot{ U}_k|^2-(|h|^2+{Ric}(\nu,\nu))\,\dot{ {U}}^2_k\,\mathrm{d}\mu.
	$$
	Note that $\tilde v_k(x)=\tilde u_k(x)=0$ if  $|x|_{\bar g} \leq \lambda(\Sigma_k)$. In particular, $\{x\in M:\tilde u_k(x)\neq 0\text{ and }f(x)\neq0\}$ is empty provided that $k$ is sufficiently large. 	Using also the curvature estimates \eqref{H est} and \eqref{hcirc est}, that is, $|x|_{\bar g}\,h=O(1)$,    we check that
	\begin{itemize}
		\item[$\circ$]$H\,\tilde v_k=O(\lambda(\Sigma_k)^{-2-\tau}),$
\item[$\circ$]$H^2\,\tilde u_k\,(u_k+f+\tilde u_k)=O(\lambda(\Sigma_k)^{-2-\tau})$,
\item[$\circ$]$g(\nabla \tilde u_k,\nabla (u_k+f+\tilde u_k))=O(\lambda(\Sigma_k)^{-2-\tau})$, and
\item[$\circ$]$(|h|^2+Ric(\nu,\nu))\,\tilde u_k\,(u_k+f+\tilde u_k)=O(\lambda(\Sigma_k)^{-2-\tau})$.		
	\end{itemize}	 
It follows that
	\begin{align*}
	\frac{d^2}{ds^2}\bigg|_{s=0}| \Sigma_k(s)|_{g}&=\int_{\Sigma_k}H\,v_k+H^2\,(u_k+f)^2+|\nabla (u_k+f)|^2-(|h|^2+{Ric}(\nu,\nu))\,(u_k+f)^2\,\mathrm{d}\mu\\&\qquad +O(\lambda(\Sigma_k)^{n-3-\tau}).
	\end{align*} 
 \indent Note that $f=0$ on $\Sigma_k\setminus B_r$. By Lemma \ref{area growth estimate}, using that $2+\tau>n-1$, we have
	$$
	\int_{\Sigma_k\setminus B_{r}} {Ric}(\nu,\nu)\,u_k^2\,\mathrm{d}\mu=O(1)\,\int_{\Sigma_k\setminus B_r}|x|_{\bar g}^{-2-\tau}\,\mathrm{d}\bar \mu=O(\lambda(\Sigma_k)^{n-3-\tau})+O(r^{n-3-\tau}).
	$$
\indent 	Similarly, using also Lemma \ref{geometric expansions}, the curvature estimates \eqref{H est} and \eqref{hcirc est}, and \eqref{spt f}, we check that
	\begin{align*} 
	&\int_{\Sigma_k\setminus B_{r}}H\,v_k+H^2\,u_k^2+|\nabla u_k|^2-|h|^2\,u_k^2\,\mathrm{d}\mu\\&\qquad =\int_{\Sigma_k\setminus B_{r}}\bar H\,\bar v_k+\bar H^2\,\bar u_k^2+|\bar\nabla\bar u_k|_{\bar g}^2-|\bar h|^2_{\bar g}\,\bar u_k^2\,\mathrm{d}\bar \mu
 + O(\lambda(\Sigma_k)^{n-3-\tau})+O(r^{n-3-\tau})
	\end{align*} 
where $\bar u_k,\,\bar v_k \in C^\infty(\Sigma_k)$ are  given by
$
\bar u_k=\bar g(e_n,\bar\nu)$ and $ \bar v_k=-\bar h(e_n^{\bar\top},e_n^{\bar\top}).
$ For instance, $$H\,v_k=\bar H\,\bar v_k+O(|x|_{\bar g}^{-2-\tau})$$ and 
$$
\int_{\Sigma_k\setminus B_r}|x|_{\bar g}^{-2-\tau}\,\mathrm{d}\bar\mu=O(\lambda(\Sigma_k)^{n-3-\tau})+O(r^{n-3-\tau}). 
$$\indent 
 Since $\Sigma_k\to\Sigma$ locally smoothly, we obtain, using also the integration by parts formula from Lemma \ref{euclidean lemma},	
$$
\int_{\Sigma_k\setminus B_{r}}\bar H\,\bar v_k+\bar H^2\,\bar u_k^2+|\bar\nabla\bar u_k|_{\bar g}^2-|\bar h|^2_{\bar g}\,\bar u_k^2\,\mathrm{d}\bar \mu
=O(1)\,\int_{\partial B_{r}\cap \Sigma} |\bar h(\Sigma)|_{\bar g}\,|e_n^{\bar\top}|_{\bar g}\,\mathrm{d}\bar l.
$$
\indent By Proposition \ref{graph estimate},  $\Sigma\setminus B_r$ is contained in the graph of a function $\psi\in C^\infty(\mathbb{R}^{n-1})$ that satisfies
	\begin{align*}  
	|y|^{-1}_{\bar g}\,|\psi(y)-a|+|\bar\nabla \psi|_{\bar g}+|y|_{\bar g}\,|\bar\nabla^2 \psi |_{\bar g}=O(|y|_{\bar g}^{-\tau/2 })
\end{align*}   
for some $a\in\mathbb{R}$. 
Consequently,  $|\bar h(\Sigma)|_{\bar g}=O(|y|_{\bar g}^{-1-\tau/2})$ and $e_n^{\bar \top}=O(|y|_{\bar g}^{-\tau/2})$ as $y\to\infty$. Moreover, $| \partial B_r\cap \Sigma|_{\bar g}=O(r^{n-2})$ as $r\to\infty$. Thus, 
\begin{align*} 
\int_{\partial B_{r}\cap \Sigma} |\bar h(\Sigma)|_{\bar g}\,|e_n^{\bar\top}|_{\bar g}\,\mathrm{d} \bar l=O(r^{n-3-\tau}).
\end{align*} 
 \indent The assertion follows from the above estimates using \eqref{tau condition}, that $H(\Sigma_k)=o(1)$ by \eqref{H est}, and that $\Sigma_k\to\Sigma$ and  $u_k\to u$ locally smoothly.
\end{proof}

\begin{prop} \label{strictly stable} 
Let $f\in C^\infty_c(\Sigma)$.
There holds
$$
\int_{\Sigma} (|h|^2+{Ric}(\nu,\nu))\,(1+f)^2\,\mathrm{d}\mu\leq \int_{\Sigma}|\nabla f|^2\,\mathrm{d}\mu.
$$
\end{prop}
\begin{proof}
We may assume that the support of $f$ is disjoint from $\partial M$, which is stable. Letting $r \to\infty$ in Proposition \ref{gamma limit lemma} and using Corollary \ref{l2 second fundamental form} and the dominated convergence theorem, we conclude that 
 \begin{align} \label{fk}  
 \int_{\Sigma} (|h|^2+{Ric}(\nu,\nu))\,(u+f)^2\,\mathrm{d}\mu\leq \int_{\Sigma}|\nabla (u+f)|^2\,\mathrm{d}\mu
 \end{align} 
 for all $f\in C^\infty_c(\Sigma)$. 
Fix  a function $\rho\in C^\infty_c(\mathbb{R})$ with $\rho(t)=1$ if $t\leq 1$. Let $f_1,\,f_2,\ldots\in C_c^\infty(\Sigma)$ be given by
 $$
 f_k(x)=\begin{dcases}1-u(x)+f(x) \qquad &\text{if } x\in B_1,\\
 \rho(k^{-1}\,|x|_{\bar g})\,(1-u(x))+f(x)&\text{else}.
 \end{dcases}.
 $$ Note that, on $\Sigma\setminus B_1$,  
 \begin{align} \label{dominated convergence} 
 |\nabla f_k|^2\leq O(1)\,\big( |x|_{\bar g}^{-2}\,|u-1|^2+|\nabla u|^2+ |\nabla f|^2\big).
 \end{align} 
 By Proposition \ref{graph estimate}, $\Sigma\setminus B_{r_0}$ is contained in the graph of a function $\psi\in C^\infty(\mathbb{R}^{n-1})$ satisfying
 \begin{align*}  
 	|y|^{-1}_{\bar g}\,|\psi(y)-a|+|\bar\nabla \psi|_{\bar g}+|y|_{\bar g}\,|\bar\nabla^2 \psi |_{\bar g}=O(|y|_{\bar g}^{-\tau/2 })
 \end{align*}   
for some $a\in\mathbb{R}$ provided that $r_0>2$ is sufficiently large. Consequently,  
 $$|u-1|+|x|_{\bar g}\,|\nabla u|=O(|x|_{\bar g}^{-\tau/2}).
 $$ 
It follows that the right side of \eqref{dominated convergence} is integrable. The assertion  now follows from applying \eqref{fk} with $f_k$ in place of $f$, letting $k\to\infty$, and using the dominated convergence theorem.
\end{proof}
\begin{rema}
	Proposition \ref{strictly stable} holds for all $f\in H^1_{loc}(\Sigma)$ for which there are $f_1,\,f_2,\hdots\in C^\infty_c(\Sigma)$ with
	$
	\nabla f_k\to\nabla f\text{ in } L^2(\Sigma)$, as can be seen using the Sobolev inequality.
\end{rema}
The following lemma is essentially due to R.~Schoen and S.-T.~Yau \cite{SchoenYau} and  A.~Carlotto \cite{Carlottocalcvar}.
\begin{lem} \label{tilde flat}
	Let $\tilde \Sigma \subset M$ be a noncompact connected two-sided properly embedded hypersurface with $\partial \tilde \Sigma=\emptyset$. Suppose that $\tilde \Sigma$  is a stable minimal surface satisfying 
	\begin{align} \label{tilde swrac} 
	\int_{\Sigma} (|h|^2+{Ric}(\nu,\nu))\,(1+f)^2\,\mathrm{d}\mu\leq \int_{\Sigma}|\nabla f|^2\,\mathrm{d}\mu
	\end{align}  
	for every $f\in C^\infty_c(\tilde \Sigma)$ and 
	\begin{align} \label{tilde areagrowth}   
		\limsup_{r\to\infty} \frac{|B_r\cap\tilde \Sigma|_{\bar g}}{\omega_{n-1}\,r^{n-1}}=1.
	\end{align} 
Suppose that $R\geq 0$ along $\tilde \Sigma$. 	Then 	$\tilde \Sigma$ is totally geodesic and isometric to flat $\mathbb{R}^{n-1}$. Moreover, there holds $R={Ric}(\nu,\nu)=0$ along the noncompact component of $\tilde \Sigma$. 
\end{lem}
\begin{proof}
	By \eqref{tilde areagrowth}, $\tilde \Sigma$ has only one end. 
	According to Lemma \ref{vanishing mass}, $(\tilde \Sigma,g|_{\tilde \Sigma})$ is asymptotically flat with mass zero. Using this and arguing as in the proof of \cite[Theorem 2]{Carlottocalcvar} using \eqref{tilde swrac}, we see that the scalar curvature of   $(\tilde \Sigma,g|_{\tilde \Sigma})$ vanishes.   By the rigidity statement of the positive mass theorem \cite[Lemma 3 and Proposition 2]{schoenconformal},  $(\tilde \Sigma,g|_{\tilde \Sigma})$ is isometric to flat $\mathbb{R}^{n-1}$. Using \eqref{tilde swrac} with $f=0$, we see that 
	$$
	\int_{\tilde \Sigma} |h|^2+{Ric}(\nu,\nu)\,\mathrm{d}\mu\leq 0.
	$$
	Using the Gauss equation and that $\tilde \Sigma$ is scalar flat, we have
	$$
	|h|^2+{Ric}(\nu,\nu)=\frac12\,(|h|^2+R).
	$$
	In conjunction with $R\geq 0$, we obtain that $h=0$ and that $R={Ric}(\nu,\nu)=0$ along $\tilde \Sigma$.
\end{proof}
\begin{rema}
	The argument in \cite[Theorem 2]{Carlottocalcvar} referred to in the proof of Lemma \ref{tilde flat} is carried out in the special case where $(M,g)$ is asymptotic to Schwarzschild. However, as the relevant part of this argument is  based on the results in \cite[\S1]{bartnik} and the positive mass theorem \cite{SchoenYau}, it applies verbatim to our setting here. 
\end{rema}
\begin{prop} Suppose that $R\geq 0$ along the noncompact component of $\Sigma$. Then \label{flat flat} 
the noncompact component of	$\Sigma$ is totally geodesic and isometric to flat $\mathbb{R}^{n-1}$. Moreover, there holds $R={Ric}(\nu,\nu)=0$ along the noncompact component of $\Sigma$. 
\end{prop}
\begin{proof}
	This follows from Proposition \ref{strictly stable} and Lemma \ref{tilde flat}.
\end{proof}

\section{Proof of Theorem \ref{main result 2}}
In this section, $(M,g)$ is a Riemannian manifold of dimension  $3< n\leq 7$ that is asymptotically flat of rate $\tau>n-3$ and  has nonnegative scalar curvature and positive mass. \\ \indent We defer the proof of the following lemma to the end of this section.
\begin{lem} \label{metric perturbation}
 Let $p_1,\,p_2$ be points in the interior of $M$. There exists  $U$ open and compactly contained in the interior of $M$ with $p_1,\,p_2\in U$  such that the following holds. For all $r>0$ sufficiently small, there exist $W\subset U$ open with $p_1,\,p_2\in W$ and a   family $\{g_t\}_{t\in(0,1)}$ of Riemannian metrics $g_t$ on $M$ such that
	\begin{align}  
			\label{g properties 2} &\circ\qquad   g_t \to g \text{ smoothly as } t\to  0,\\
	\label{g properties 1} &\circ\qquad  g_t=g  \text{ in }  M\setminus W,\\
	\label{g properties 5} &\circ\qquad  g_t<g \text{ in }W,\text{ and}\\
	\label{g properties 3} &\circ\qquad  R(g_t)>0 \text{ in } \{x\in W:\operatorname{dist}(x,p_2)> r\}. 
\end{align} 
\end{lem}
Suppose, for a contradiction, that $\Omega_1,\,\Omega_2,\hdots\subset M$ are isoperimetric regions in $(M,g)$   with $| \Omega_k|\to\infty$ and such that $\Omega_k$ converges locally smoothly to a region $\Omega\subset M$ whose boundary $\Sigma=\partial \Omega$ is noncompact and area-minimizing.
\begin{prop}  \label{many area minimizing}  Let $p$ be a point in the interior of $M$. There exist  $U$ open and compactly contained in the interior of $M$, Riemannian metrics $\tilde g_1,\,\tilde g_2,\hdots $ on $M$, and regions  $ \Omega_p,\,\tilde \Omega_1,\,\tilde \Omega_2,\hdots \subset M$  with the following properties.
	\begin{align}
\label{a}		&\circ\qquad \operatorname{spt}(\tilde g_k-g)\subset U\text{ for all $k$ and $\tilde g_k\to g$ smoothly.}\\
\label{b}		&\circ\qquad \tilde \Omega_k\text{ is an isoperimetric region in $(M,\tilde g_k)$ and $|\tilde \Omega_k|_{\tilde g_k}\to\infty$}.\\
\label{c}		&\circ\qquad \tilde \Omega_k\to  \Omega_p\text{ locally smoothly.}\\
\label{d}		&\circ\qquad \partial \Omega_p\text{ is a noncompact area-minimizing boundary in $(M,g)$ that is stable with respect}\\ &\qquad\hspace{.4cm}\text{to asymptotically constant variations. Moreover,  $p\in \partial \Omega_p$}. \notag
	\end{align}
\end{prop}
\begin{proof}
	We first assume that $\partial M=\emptyset$. \\ \indent 
  We choose $p_1\in \Sigma$ and let $p_2=p$. Let $U\subset M$ be the  open set from Lemma \ref{metric perturbation} and $r>0$  small enough so that the conclusion of Lemma \ref{metric perturbation} holds. According to Lemma \ref{metric perturbation}, there is  $W(r)\subset U$ open and a family $\{g_t\}_{t\in(0,1)}$ of Riemannian metrics $g_t$ on $M$ satisfying  \eqref{g properties 2}, \eqref{g properties 1}, \eqref{g properties 5}, and \eqref{g properties 3}. \\ \indent Let    $t\in(0,1)$. Recall the isoperimetric profile $A_g$ defined in \eqref{isoperimetric profile}. Let  $c(t)=|W(r)|_g-|W(r)|_{g_t}$. Note that
  \begin{itemize}
  	\item[$\circ$] $c(t)>0$ by \eqref{g properties 5} and
  	\item[$\circ$] $c(t)=o(1)$ as $t\to  0$ by \eqref{g properties 2}.
  \end{itemize}
  Since $\partial \Omega_k\to \Sigma$ locally smoothly and $W(r)\cap\Sigma\neq \emptyset$, it follows from $\eqref{g properties 5}$ that  there is $\varepsilon(t)>0$ such that, for all $k$ sufficiently large,
    \begin{align} \label{area comparison} |\partial \Omega_k|_{g_t}+\varepsilon(t)<|\partial \Omega_k|_{g}.
     \end{align} 
   Using Lemma \ref{profile properties}, we see that
	\begin{align} \label{profile growth estimate}
		|\partial \Omega_k|_{g}=A_g(|\Omega_k|_{g})\leq A_g(V)+o(1)
	\end{align} 
as $t\to 0$ for any amount of volume $V$ such that    $|\Omega_k|_{g}-c(t)\leq V
\leq |\Omega_k|_{g}+c(t)$. 
	By Lemma \ref{existence iso}, for every $k$ sufficiently large,  there exists an isoperimetric region $\Omega_k(t)\subset M$ with respect to $g_t$ such that $|\Omega_k(t)|_{g_t}=|\Omega_k|_{g_t}$. \\ \indent We claim that there is  $\tilde W(r)\Subset W(r)$ open such that, for all $k$ sufficiently large,
	\begin{align} \label{non empty intersection} \tilde W(r)\cap \partial \Omega_k(t)\neq\emptyset. 
	\end{align}  
	 To see this, assume the contrary. Passing to a subsequence and using \eqref{g properties 1}, we have $g=g_t+o(1)$ on $\partial \Omega_k\cap  W(r)$ as $k\to\infty$. Moreover, using Lemma \ref{area growth estimate}, we have $|\partial \Omega_k\cap W(r)|_g=O(1)$ as $k\to\infty$. By \eqref{g properties 1}, we have $|\partial \Omega_k\setminus  W(r)|_{g_t}=|\partial \Omega_k\setminus  W(r)|_{g}$.   It follows that 
	\begin{align} \label{comparison} 
	|\partial \Omega_k(t)|_{g_t}=|\partial \Omega_k(t)|_{g}+o(1)
\end{align} 
as $k\to\infty$. 	Using \eqref{g properties 5}, we have
	$$
	|\Omega_k(t)|_g\leq |\Omega_k(t)|_{g_t}+c(t)=|\Omega_k|_{g_t}+c(t)\leq |\Omega_k|_g+c(t).
	$$
	Similarly,
	$
	|\Omega_k(t)|_g\geq  |\Omega_k|_g-c(t).
	$
Using \eqref{profile growth estimate} with $V=|\Omega_k(t)|_g$, we conclude that $$|\partial \Omega_k(t)|_{g}\geq |\partial \Omega_k|_{g}-o(1)$$ as $k\to\infty$. This is not compatible with \eqref{area comparison}, \eqref{comparison}, and the fact that $\Omega_k(t)$ is isoperimetric with respect to $g_t$. This establishes \eqref{non empty intersection}. \\ \indent 
	By Proposition \ref{smooth local convergence}, $ \Omega_k(t)$ converges locally smoothly as $k\to\infty$ to a region $\Omega(t)$ with noncompact boundary $\Sigma(t)$ that is area-minimizing with respect to $g_t$. Using \eqref{non empty intersection}, we see that $\Sigma(t)$ intersects the closure of  $\tilde W(r)$. 
	Using Proposition  \ref{flat flat} and \eqref{g properties 3}, it follows that, in fact,   $ \{x\in M:\operatorname{dist}(x,p)<r\}\cap \Sigma(t)\neq \emptyset$. \\ \indent 
	Passing to a subsequence, we see that, as $t\to 0$, $\Omega(t)$ converges locally smoothly to a  region $\Omega_p(r)\subset M$ with noncompact boundary $\Sigma_p(r)$ that is  area-minimizing with respect to $g$  such that $$\{x\in M:\operatorname{dist}(x,p)\leq r\}\cap \Sigma_p(r)\neq\emptyset.$$ Moreover, passing to a subsequence, we see that, as $r\to 0$, $\Omega_p(r)$ converges locally smoothly to a region $\Omega_p$ with noncompact boundary $\Sigma_p$ that is area-minimizing with respect to $g$ and such that $p\in  \Sigma_p$. \\ \indent  
	By a diagonal argument, using that $W(r)\subset U$ for every $r>0$ sufficiently small, we see that there are Riemannian metric $\tilde g_1,\,\tilde g_2,\hdots $ on $M$ and regions  $ \,\tilde \Omega_1,\,\tilde \Omega_2,\hdots\subset M$ satisfying \eqref{a}, \eqref{b}, and \eqref{c}. By \eqref{a}, the asymptotic behavior of $\tilde g_k$ does not depend on $k$. We can therefore repeat the argument that led to Proposition \ref{strictly stable} to see that \eqref{d} holds.  This completes the proof in the case where $\partial M=\emptyset$.  \\ \indent
In the case where $\partial M\neq\emptyset$, we note that $\partial M$ is also an outermost minimal surface with respect to $g_t$ for all sufficiently small $t>0$. As we explain in Lemma \ref{lem:locallyisoperimetricboundary}, this ensures that all bounded components of a noncompact area-minimizing boundary in $(M,g_t)$   are contained in $\partial M$. The rest of the proof therefore only requires formal modifications. 
\end{proof}
\begin{proof}[Proof of Theorem \ref{main result 2}] 
	Let $(M,g)$ be asymptotically flat of rate $\tau>n-3$ with nonnegative scalar curvature. Suppose that there are isoperimetric regions $\Omega_1,\,\Omega_2,\hdots\subset M$ of $(M,g)$ with $|\Omega_k|\to \infty$
	 such that $\Omega_k$ converges locally smoothly to a region $\Omega\subset M$ whose boundary $\Sigma=\partial \Omega$ is noncompact and area-minimizing.
	\\ \indent
	Suppose, for a contradiction, that $(M,g)$ has positive mass $m$. \\ \indent 
	We first assume that $\partial M=\emptyset$. \\ \indent 
Our goal is to show that the curvature tensor ${Rm}$ vanishes everywhere.\\ \indent  Let $p,\,p_1\in M$  be such that $\operatorname{dist}(p_1,p)=1$. Applying Proposition \ref{many area minimizing} to $p_1$ and using  Lemma \ref{tilde flat}, we obtain a noncompact area-minimizing boundary $\Sigma_{p_1}$ with $p_1\in \Sigma_{p_1}$ that is isometric to flat $\mathbb{R}^{n-1}$, totally geodesic, and such that $R=Ric(\nu,\nu)=0$ along $\Sigma_{p_1}$. Next, applying Proposition \ref{many area minimizing} to a point $p_2\in M$ with $\operatorname{dist}(p_2,p)=1/2$ and $p_2\notin \Sigma_{p_1}$. Using also  Lemma \ref{tilde flat}, we obtain a noncompact area-minimizing boundary $\Sigma_{p_2}\neq \Sigma_{p_1}$ with $p_2\in \Sigma_{p_2}$ that is isometric to flat $\mathbb{R}^{n-1}$, totally geodesic, and such that $R=Ric(\nu,\nu)=0$ along $\Sigma_{p_2}$. Iterating this procedure and using standard convergence results from geometric measure theory, we see that there are  distinct connected noncompact   area-minimizing boundaries $\Sigma_p,\,\Sigma_1,\,\Sigma_2,\hdots\subset M$ such that
	\begin{itemize}
		\item[$\circ$] $\Sigma_k$ is isometric to flat $\mathbb{R}^{n-1}$ for all $k$,
		\item[$\circ$] $\Sigma_k$ is totally geodesic for all $k$, 
		\item[$\circ$] $\Sigma_k\to \Sigma_p$ locally smoothly,
	\end{itemize}
and $\Sigma_p$ is isometric to flat $\mathbb{R}^{n-1}$, totally geodesic, and such that $R=Ric(\nu,\nu)=0$ along $\Sigma_{p}$.
	   Let $W,\,X,\,Y,\,Z$ be tangent fields along $\Sigma_p$. Note that
\begin{align} \label{first flat}  
{Rm}(W,\,X,\,Y,\,Z)=0.
\end{align} 
Indeed, this follows from the Gauss equation, using that $\Sigma_p$ is isometric to flat $\mathbb{R}^{n-1}$ and totally geodesic. Similarly, using the Codazzi equation, we obtain
\begin{align} \label{second flat}  
{Rm}(W,\,X,\,Y,\nu)=0.
\end{align} 
Since $\Sigma_p$ and $\Sigma_k$ are totally geodesic and distinct, their intersection is either empty or  transverse. In the latter case, ${Rm}=0$ along the intersection. We may therefore assume that there is a neighborhood $U$ of $p$ such that $U\cap \Sigma_p\cap \Sigma_k=\emptyset$ for all $k\geq 1$. \\ \indent 
Since $\Sigma_p$ and $\Sigma_k$ are totally geodesic, the components of $\Sigma_p\cap \Sigma_k$ are totally geodesic and therefore hyperplanes of $\Sigma_p$. Since $\Sigma_p$ and  $\Sigma_k$ are embedded, the components of $\Sigma_p\cap \Sigma_k$ must be parallel. It follows that, passing to a subsequence, there are three cases; cf.~\cite[Lemma 1]{liu}.
\begin{itemize}
	\item[$\circ$] For every $k$, $\Sigma_p\cap \Sigma_k=\emptyset$.
	\item[$\circ$] For every $k$, $\Sigma_p\cap \Sigma_k$ consists of a single hyperplane in $\Sigma_p$.
	\item[$\circ$] For every $k$, $\Sigma_p\cap \Sigma_k$ consists of at least two parallel hyperplanes in $\Sigma_p$.
\end{itemize}
In the second and the third case, let $\Sigma^0_k$ be the closure of the component of $\Sigma_p\setminus \Sigma_k$ that contains $p$. Note that $\Sigma^0_k$ is isometric to a half-space in the second case and isometric to a slab in the third case. Since $U\cap \Sigma_p\cap \Sigma_k=\emptyset$, it follows that
$
\liminf_{k\to\infty}\operatorname{dist}(p,\partial \Sigma^0_k)>0.
$\\\indent
In the first case, we may argue exactly as in \cite[p.~993]{CCE}. Specifically, we may represent $\Sigma_k$ as the graph of a positive function $\psi_k$ over larger and larger domains, exhausting  $\Sigma_p$ as $k\to\infty$. By the Harnack inequality \cite[\S8.8]{gilbargtrudinger}, $\psi_k$ is bounded by a multiple of $\psi_k(p)$ locally in $\Sigma_p$ as $k\to\infty$. 
Using the first variation of the second fundamental form and proceeding as in \cite[p.~333]{simonstrict},  
we obtain a positive function $f\in C^\infty(\Sigma_p)$  such that
\begin{align} \label{f PDE} 
	\nabla_{X,Y}^2f+{Rm}(X,\nu,\nu,Y)\,f=0
\end{align} 
for all tangent fields $X,\,Y$ of $\Sigma_p$. Tracing \eqref{f PDE} and using that ${Ric}(\nu,\nu)=0$ along $\Sigma_p$, we see that $f$ is harmonic. By the Liouville theorem,  $f$ is equal to a  constant. It follows that ${Rm}(X,\nu,\nu,Y)=0$. In conjunction with \eqref{first flat} and \eqref{second flat}, we conclude that ${Rm}=0$ along $\Sigma_p$ and in particular at $p$.
\\ \indent 
In the second case, if $\limsup_{k\to\infty}\operatorname{dist}(p,\partial \Sigma^0_k)=\infty$, we may argue as in the first case. If $\operatorname{dist}(p,\partial \Sigma^0_k)=O(1)$, then, passing to a subsequence, we may assume that $\Sigma_k^0$ converges locally smoothly to a half-space $\Sigma^0\subset \Sigma_p$. As before, we may represent $\Sigma^0_k$ as the graph of a smooth function $\psi_k$ over larger and larger domains, exhausting $\Sigma^0$ as $k\to\infty$.  Arguing as in the first case, using also the Boundary Harnack inequality (see, e.g., \cite[Theorem 11.5]{caffarelli}), we obtain a harmonic function $f\in C^\infty(\Sigma^0)$ that satisfies \eqref{f PDE} in $\Sigma^0$, $f=0$ on $\partial \Sigma^0$, and $f>0$ away from $\partial \Sigma^0$. By \cite[Theorem I]{shortproofs}, $f$ is a linear function. As before, it follows that ${Rm}=0$ along $\Sigma_p$ and in particular at $p$. \\ \indent 
Finally, suppose, for a contradiction, that the third case arises and that $\operatorname{dist}(p,\partial \Sigma^0_k)=O(1)$. Passing to a subsequence, $\Sigma_k^0$ converges locally to a slab $\Sigma^0\subset \Sigma_p$. As in the previous case, we obtain a harmonic function $f\in C^\infty(\Sigma^0)$ that satisfies \eqref{f PDE} in $\Sigma^0$, $f=0$ on $\partial \Sigma^0$, and $f>0$ away from $\partial \Sigma^0$.  By \eqref{f PDE}, $\nabla^2 f=0$ along $\partial \Sigma^0$ and it follows that $|\nabla f|$ is bounded from above by a positive constant on $\partial \Sigma^0$. This is not compatible with Lemma \ref{harmonic on slab}. \\ \indent 
It follows that $Rm=0$ in $M$. Since  $(M,g)$ is asymptotically flat, it follows that $(M,g)$ is isometric to flat $\mathbb{R}^n$.  This is not compatible with $m>0$.  
This completes  the proof in the case where $\partial M=\emptyset$. \\ \indent In the case where $\partial M\neq \emptyset$ is an outermost minimal surface, Lemma \ref{lem:locallyisoperimetricboundary} shows that all bounded components of a noncompact area-minimizing boundary in $(M,g)$ are contained in $\partial M$. The proof therefore only requires formal modifications.
	\end{proof} 
\begin{proof}[Proof of Corollary \ref{main result}]
	Let $(M,g)$ be asymptotically flat of rate $\tau>n-3$  with nonnegative scalar curvature and positive mass $m$. Suppose, for a contradiction, that there are isoperimetric regions $\Omega_1,\,\Omega_2,\hdots\subset M$ of $(M,g)$ with $|\Omega_k|\to \infty$ and a compact set $K\subset M$   such that $K\cap \partial M=\emptyset$ and  $K\cap\partial \Omega_k\neq \emptyset$ for all $k$.
	\\ \indent By Proposition \ref{smooth local convergence},  $\Omega_k$ converges locally smoothly to a region $\Omega\subset M$ whose boundary $\Sigma=\partial \Omega$ is noncompact and area-minimizing. By Theorem \ref{main result 2}, $(M,g)$ is isometric to flat $\mathbb{R}^n$. This is not compatible with the assumption $m>0$.
\end{proof}
\begin{proof}[{Proof of Lemma \ref{metric perturbation}}]
Arguing as in \cite[p.~21]{chaoli}, we see that there is $r_0>0$ depending only on $(M,g)$ with the following property. Given $0<r<r_0$ and $q\in M$, there exists a function $v_{r,q}\in C^\infty(M)$ satisfying
	\begin{align}  
		\label{v properties 1} &\circ\qquad  v_{r,q}=0  \text{ in }  \{x\in M:\operatorname{dist}_g(x,q)\geq 6\,r\},\qquad\qquad\qquad\qquad\qquad\\
		\label{v properties 2} &\circ\qquad  v_{r,q}<0  \text{ in }  \{x\in M:\operatorname{dist}_g(x,q)< 6\,r\},\text{ and}\\
		\label{v properties 3} &\circ\qquad  \Delta v_{r,q}<0 \text{ in } \{x\in M:r<\operatorname{dist}_g(x,q)<6\,r\}.		
	\end{align} 
	Decreasing $r>0$ if necessary, we may choose points $q_1,\hdots,q_N\in M$ with $q_1=p_1$ and $q_N=p_2$ such that  
	\begin{align}
		\label{ball property 2} &\{x\in M:\operatorname{dist}(x,q_i)\leq r\}\subset \{x\in M:3\,r\leq \operatorname{dist}(x,q_{i+1})\leq 5\,r\} \text{ for all }i\leq N-1\text{ and}\\ \notag
		&\{x\in M:\operatorname{dist}(x,q_i)\leq 6\,r\}\text{ is disjoint from }\partial M\text{ for all }1\leq i\leq N; 
	\end{align}
see Figure \ref{perturbation}.
 	\begin{figure}\centering
	\includegraphics[width=0.7\linewidth]{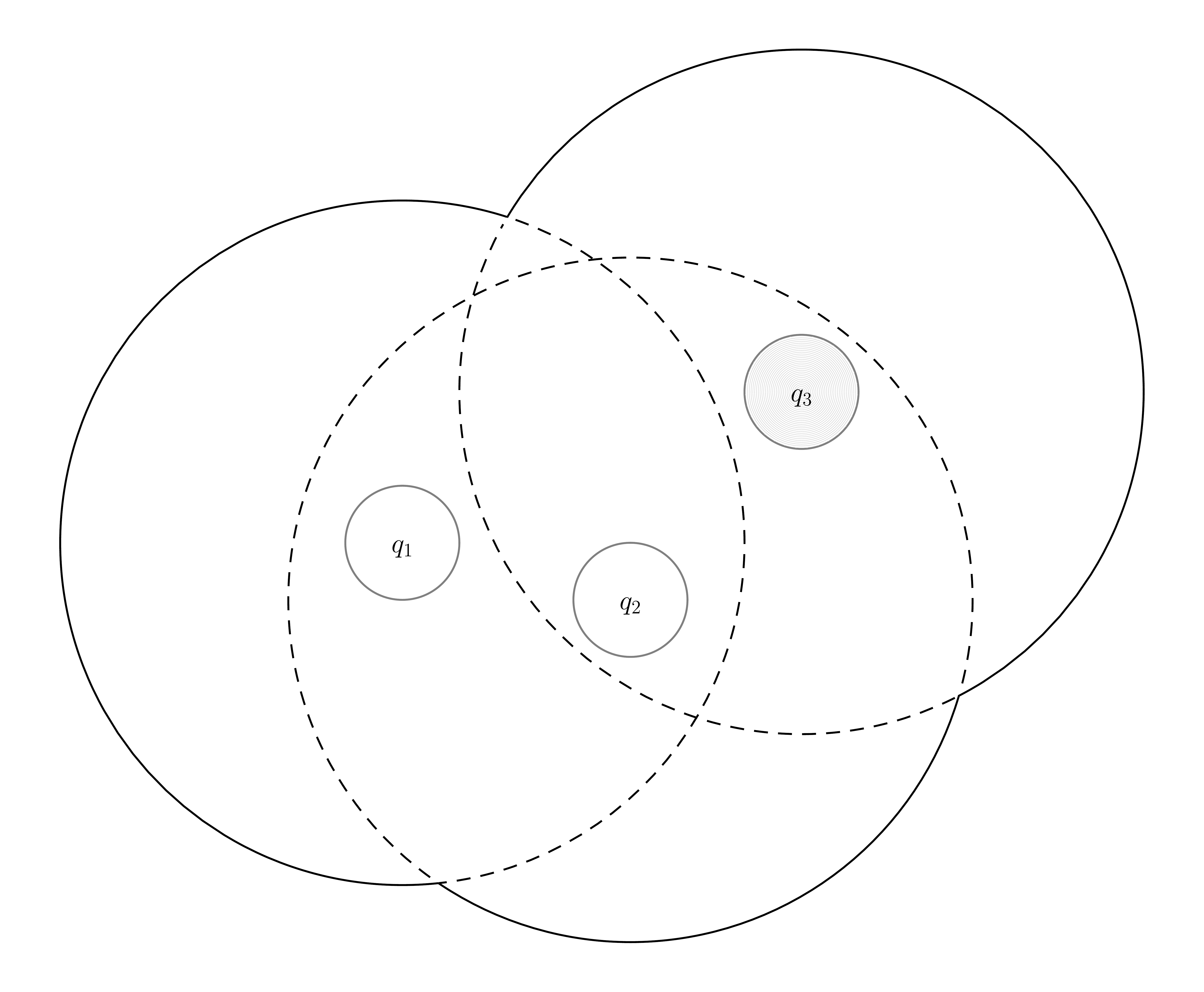}
	\caption{An illustration of the construction in Lemma \ref{metric perturbation}. The open set $W$ is bounded by the solid black line. The function $v$ is superharmonic outside of the hatched region. }
	\label{perturbation}
\end{figure}
	Define $a_1,\hdots,a_N$ recursively by $a_1=1$ and, for $i= 2,\hdots,N$,
	$$
	a_i=1+\frac{\sup\{\Delta v_{r,q_{i-1}}(x):\operatorname{dist}(x,q_{i-1})<r\}}{|\inf\{\Delta v_{r,q_i}(x):3\,r\leq \operatorname{dist}(x,q_{i})\leq 5\,r\}|}\,a_{i-1}.
	$$
	Let
	$$
	W=\bigcup_{i=1}^N\{x\in M:\operatorname{dist}(x,q_i)< 6\,r\}
	\qquad \text{and} \qquad v=\sum_{i=1}^N a_i\,v_{r,q_i}.
	$$
	Note that
	\begin{itemize}
		\item[$\circ$] $v=0$ in $M\setminus W$ by \eqref{v properties 1},
		\item[$\circ$] $v<0$ in $W$ by \eqref{v properties 2}, and 
		\item[$\circ$] $\Delta v<0$ in $\{x\in W:\operatorname{dist}(x,p_2)\geq r\}$ by \eqref{v properties 3} and \eqref{ball property 2}.
	\end{itemize}\indent  \indent 
	For all $\delta>0$ sufficiently small, the Riemannian metrics $$g_t=(1+t\,\delta\,v)^{\frac{4}{n-2}}\,g$$ where $t\in(0,1)$ satisfy  the properties asserted in the lemma.\end{proof} 
 \section{Proof of Theorem \ref{counterexample 2}}
Let $3<n\leq 7$. In this section, we use the gluing technique recently developed by Y.~Mao and Z.~Tao \cite{maotao}, improving upon the pioneering work of A.~Carlotto and R.~Schoen \cite{carlottoschoen},  to construct an example of a Riemannian manifold that is asymptotically flat of rate $n-2$ with nonnegative scalar curvature and positive mass  that admits a noncompact area-minimizing boundary that is stable with respect to asymptotically constant variations \eqref{stable wrt acv}. 
\begin{lem} \label{localization}
	There exists a Riemannian metric $g$ on $\mathbb{R}^n$ that is asymptotically flat of rate $ n-2$ with nonnegative scalar curvature and positive mass such that $$g=\bar g \text{ on } \{x\in\mathbb{R}^n:x^n\leq 0\}.$$
\end{lem}
\begin{proof}
	The Riemannian metric
	$$
	\hat g=\left[1+(1+|x|_{\bar g}^{2\,n-4})^{-\frac12}\right]^{\frac{4}{n-2}}\,\bar g
	$$
	is asymptotically flat of rate $n-2$ with nonnegative scalar curvature and  mass $m=2$. The assertion follows from localizing $(\mathbb{R}^n,\hat g)$ to the half-space $\{x\in\mathbb{R}^n:x^n>0\}$ using  \cite[Theorem 1]{maotao}.
\end{proof}
\begin{rema}
	We use \cite[Theorem 1]{maotao} instead of \cite[Theorem 2.3]{carlottoschoen} in the proof of Lemma \ref{localization} to ensure that the localized metric $g$ is again asymptotically flat of rate $n-2$.
\end{rema}

Given $\beta>0$, let  
$$
c_\beta=\int_{\beta }^\infty \frac{1}{\sqrt{\beta^{-3}\,s^3-1}}\,\mathrm{d}s
$$
and note that $c_\beta<\infty$. 
Let $ f_\beta\in C^\infty((\beta,\infty))\cap C^0([\beta,\infty))$ be given by 
$$
 f_\beta (t)=\int_{\beta }^t \frac{1}{\sqrt{\beta^{-3}\,s^3-1}}\,\mathrm{d}s
$$ 
and $S_{\beta,z}\subset \mathbb{R}^n$ be the smooth hypersurface with boundary given by 
 \begin{align} \label{barriers2} S_{\beta,z}=\{(y,z+ c_\beta- f_\beta (|y|)) :y\in\mathbb{R}^{n-1}\text{ and }|y|_{\bar g}\geq \beta \}.\end{align}   We orient $S_{\beta,z}$ by the unit normal pointing in direction of $e_n$ in $S_{\beta,z}\setminus \partial S_{\beta,z}$. Note that, on $\partial S_{\beta,z}$, 
 \begin{align} \label{bc}
 \bar g(\bar \nu(S_{\beta,z}),e_n)=0.
 \end{align}
\begin{lem} \label{catenoidal barriers}
There is $\beta_*>1$ such that	$S_{\beta_*,z}$ is strictly mean-convex for every $z\in\mathbb{R}$.   
\end{lem}    
\begin{proof}
	At a point $x\in S_{\beta,z}$ with $x^n=z+c_\beta- f_\beta (t)$ for some $t\in(\beta,\infty)$, the Euclidean principal curvatures $\kappa_1,\ldots\kappa_{n-1}$ of $S_{\beta,z}$ are given by 
$$
\kappa_1,\ldots,\kappa_{n-2}=\beta^{3/2}\,t^{-5/2}\qquad\text{and}\qquad \kappa_{n-1}=-\frac32\,\beta^{3/2}\,t^{-5/2}.
$$
	It follows that 
	$$
	\bar H(S_{\beta,z})=\left(n-\frac{7}{2}\right)\,\beta^{3/2}\,t^{-5/2}
\qquad\text{and}\qquad 
	|\bar h(S_{\beta,z})|_{\bar g}=\sqrt{n+\frac14}\,\beta^{3/2}\,t^{-5/2}. 
	$$
	Using that $|x|_{\bar g}\geq t$ and Lemma \ref{geometric expansions}, we conclude that 
	$$
	H(S_{\beta,z})=\left(n-\frac{7}{2}\right)\,\beta^{3/2}\,t^{-5/2}
+O(1)\,t^{1-n}+O(1)\,\beta^{3/2}\,t^{-1/2-n}.
	$$
	Since $n>7/2$, we conclude that $H(S_{\beta,z})>0$ provided that $\beta>1$ is sufficiently large. 
\end{proof}
Note that there is $r_*>1$ such that, for every $r>r_*$, $S^{n-2}_r(0)\times \mathbb{R}$ is strictly mean-convex in $(\mathbb{R}^n,g)$ when oriented by the normal pointing away from the vertical axis. Given  $r>r^*$ and $z<0$,  let $ \Sigma_{r,z}\subset \mathbb{R}^n$ be the least area hypersurface in $(\mathbb{R}^n, g)$ with $$\partial  \Sigma_{r,z}= S^{n-2}_r(0)\times\{z\}.$$ 
By the strong maximum principle, $\Sigma_{r,z}$ is a vertical graph with finite slope near $\partial \Sigma_{r,z}$. 
\begin{lem}  \label{smp}
	Let $r>r_*$ and $z<0$. There holds  $\Sigma_{r,z}\subset \mathbb{R}^{n-1}\times[z,\infty)$ and $\Sigma_{r,z}\cap S_{\beta_*,z}=\emptyset$ where $\beta_*$ is as in Lemma \ref{catenoidal barriers}.
\end{lem}
\begin{proof}
	Let $r>1$ and $z<0$. 
	Using that $g=\bar g$ in $\{x\in \mathbb{R}^n:x^n\leq 0\}$, we see that the planes $\{x\in \mathbb{R}^n:x^n=z'\}$ are minimal in $(\mathbb{R}^n,g)$ for every $z'\leq 0$. 
	By the strong maximum principle, we conclude that $\Sigma_{r,z}\cap \{x\in \mathbb{R}^n:x^n=z'\}=\emptyset$ for every $z'<z$. Similarly, by the strong maximum principle, using Lemma \ref{catenoidal barriers} and \eqref{bc},   we conclude that $\Sigma_{r,z}\cap S_{\beta^*,z'}=\emptyset$ for every $z'\geq z$. The assertion follows.
\end{proof}
\begin{lem} \label{rz}
There holds $\Sigma_{r,z}\setminus \partial \Sigma_{r,z}=B^{n-1}_r(0)\times\{z\}$ provided that $z<-4\,\beta_*-4\,c_{\beta_*}$ and $r>\beta_*+r_*$.
\end{lem}
\begin{proof}
	Let $z<-4\,\beta_*-4\,c_{\beta_*}$ and $r>\beta_*+r_*$. 
	Since $\Sigma_{r,z}$ is area-minimizing, we have \begin{align} \label{basic area estimate} |\Sigma_{r,z}|\leq |B^{n-1}_r(0)\times\{z\}|_{\bar g}= \omega_{n-1}\,r^{n-1}.\end{align} Let $\Sigma^-_{r,z}=\{x\in \Sigma_{r,z}: x_n \leq z+c_{\beta_*} \}$ and $\Sigma^+_{r,z}=\{x\in \Sigma_{r,z}: x_n > z+c_{\beta_*} \}$. Note that $g=\bar g$ near $\Sigma^-_{r,z}$.  By Lemma \ref{smp}, $\partial \Sigma^-_{r,z}\setminus \partial \Sigma_{r,z}\subset B^{n-1}_{\beta_*}(0)\times\{z+c_{\beta_*}\}$. Using this and the intermediate value theorem, we obtain
	$$
	|\Sigma_{r,z}^-|=|\Sigma_{r,z}^-|_{\bar g}\geq \omega_{n-1}\,r^{n-1}-\omega_{n-1}\, \beta_*^{n-1}.
	$$ 
	 Suppose, for a contradiction, that there is $x\in \Sigma_{r,z}$ with $x^n=z/2$. It follows that  $B^n_{-z/4}(x)\cap \Sigma_{r,z}\subset \Sigma_{r,z}^+$ so that, by the monotonicity formula \eqref{monotonicity formula}, using that $g=\bar g$ in $B^n_{-z/4}(x)$, 
	$$
|\Sigma^+_{r,z}|\geq |B^n_{-z/4}(x)\cap \Sigma_{r,z}|=|B^n_{-z/4}(x)\cap \Sigma_{r,z}|_{\bar g}\geq 4^{1-n}\,\omega_{n-1}\,z^{n-1}. 	
	$$
Consequently, $|\Sigma^-_{r,z}|+|\Sigma^+_{r,z}|>n\,\omega_n\,r^{n-1}$. This is incompatible with \eqref{basic area estimate}. It follows that $\Sigma_{r,z}\subset\{x\in \mathbb{R}^n:x^n\leq z/2\}$. Using that $g=\bar g$ in $\{x\in \mathbb{R}^n:x^n \leq 0\}$ and that $\Sigma_{r,z}$ is area-minimizing, we conclude that $\Sigma_{r,z}\setminus \partial \Sigma_{r,z}=B^{n-1}_r(0)\times\{z\}$.
\end{proof}
\begin{proof}[Proof of Theorem \ref{counterexample 2}]
Let $z<-4\,\beta_*-4\,c_{\beta_*}$. By Lemma \ref{rz}, $\mathbb{R}^{n-1}\times\{z\}$ is area-minimizing in $(\mathbb{R}^n,g)$. Since $g=\bar g$ in $\{x\in\mathbb{R}^n:x^n\leq 0\}$,   $\mathbb{R}^{n-1}\times \{z\}$ is stable with respect to asymptotically constant variations.
\end{proof}
\begin{rema} \label{schoenyaucarlotto}
	As asserted by R.~Schoen and S.-T.~Yau in \cite[p.~1]{pnas}, the catenoid-like barriers \eqref{barriers2} can be used to prove that, in an asymptotically flat Riemannian manifold of dimension $3<n\leq 7$, there exist area-minimizing hypersurfaces that are asymptotic to any given hyperplane. We are grateful to R.~Schoen for pointing this out to us. In view of \cite[Theorem 1 and Theorem 2]{Carlottocalcvar}, these hypersurfaces do not necessarily satisfy \eqref{stable wrt acv}. Barriers of the form \eqref{barriers2} have also been used by R.~Schoen and S.-T.~Yau in their proof of the spacetime positive mass theorem; see \cite[p.~248]{SchoenYau2}.
\end{rema}
\begin{appendices}

\section{Asymptotically flat manifolds} \label{adm appendix}
In this section, we recall background on asymptotically flat manifolds. \\ \indent 
Let $(M, g)$ be a connected complete Riemannian manifold of dimension $n \geq 3$. \\
We say that $(M, g)$ is asymptotically flat of rate $\tau > (n-2)/2$ if there is a nonempty compact set $K \subset M$ and a diffeomorphism
\begin {equation} \label{chartatinfinity}
 \{ x \in \mathbb{R}^n : |x|_{\bar g} > 1/2\}\to M \setminus K
\end {equation}
such that, in the corresponding coordinate system,
\[
|g - \bar g|_{\bar g} + |x|_{\bar g} \, |\bar D g|_{\bar g} + |x|_{\bar g}^2\, |\bar D^2 g|_{\bar g} = O (|x|_{\bar g}^{-\tau}).
\]
Here, $\bar g$ is the Euclidean metric on $\mathbb{R}^n$,  $\bar D$ is the Euclidean derivative, and a bar indicates that a geometric quantity is computed with respect to $\bar g$. In addition, we require the scalar curvature  of $(M,g)$ to be integrable. If $\partial M\neq \emptyset$, we require that the boundary is a compact minimal surface that is outermost in the sense that every closed minimal hypersurface in $(M, g)$ is contained in the boundary. \\ \indent The particular choice of diffeomorphism \eqref{chartatinfinity} is usually fixed and referred to as the chart at infinity of $(M, g)$. Given $r > 1/2$, we use $S_r$ to denote the subset of $M$ corresponding to $S^{n-1}_r(0)=\{x\in\mathbb{R}^n:|x|_{\bar g}=r\}$ in this chart and $B_r$ to denote the bounded component of $M \setminus S_r$.
\\ \indent 
If $(M, g)$ is asymptotically flat, the quantity
\begin{align} \label{def adm mass} 
	m=\frac{1}{2\,(n-1)\,n\,\omega_{n}}\,\lim_{\lambda\to\infty}\lambda^{-1}\,\int_{S^{n-1}_{\lambda}(0)}\sum_{i,\,j=1}^nx^i\,\left[(\bar D_{e_j}g)(e_i,e_j)-(\bar D_{e_i}g)(e_j,e_j)\right]\,\mathrm{d}\bar\mu
\end{align} 
is called the mass of $(M, g)$; see \cite{ADM}. 	Here, $e_1,\dots,e_n$ is the standard basis of $\mathbb{R}^n$.  The existence of the limit in \eqref{def adm mass} follows from the integrability of the scalar curvature and the decay assumptions on $g$. It is independent of the particular choice of chart at infinity; see \cite[Theorem 4.2]{bartnik}. 

\section {Isoperimetric regions and their limits}
\label{iso appendix}
In this section, we collect results on large isoperimetric regions. \\ \indent 
Let $(M, g)$ be a connected complete Riemannian manifold of dimension $3 \leq n \leq 7$ that is asymptotically flat. Recall that, by definition, $\partial M$ is either empty or an outermost minimal surface. 

Let $\hat M$ be an open manifold in which $M$ is embedded.

A subset $\Omega \subset M$ is called a region if 
\[
\hat \Omega = \Omega \cup (\hat M \setminus M)
\] 
is a properly embedded top-dimensional submanifold of $\hat M$. Note that  $\partial \hat \Omega\subset M$ is a properly embedded hypersurface. It does not depend on the  choice of extension $\hat M$ and will be denoted by $\partial \Omega$.

Let $\Omega \subset M$ be a region. The second fundamental form $h$ and the mean curvature scalar $H$ of $\partial \Omega$ are computed with respect to the normal pointing out of $\Omega$. 

We are interested in three special types of regions in this paper.

\begin{enumerate} [$\circ$] 
	\item A region $\Omega \subset M$ is isoperimetric if it is compact and
	\[
	|\partial \Omega| \leq |\partial \tilde \Omega|
	\]
	for every compact region $\tilde \Omega \subset M$ with $|\tilde \Omega| = | \Omega|$.
	\item A region $\Omega \subset M$ has area-minimizing boundary if, for every $U \Subset M$ open, there holds
	\[
	|U \cap \partial \Omega | \leq |U \cap \partial \tilde \Omega|
	\]
	for every region $\tilde \Omega \subset M$ with $\tilde \Omega \, \triangle \, \Omega \Subset U$.
	\item 
	A region $\Omega \subset M$ is locally isoperimetric if, for every $U \Subset M$ open, there holds
	\[
	|U \cap \partial \Omega | \leq |U \cap \partial \tilde \Omega|
	\]
	for every region $\tilde \Omega \subset M$ with $\Omega \, \triangle \, \tilde \Omega \Subset U$ and $|U \cap \tilde \Omega| = |U \cap \Omega|$.
\end{enumerate}

\begin {rema}  \label{regularity remark}
Alternatively, we could introduce these notions using sets with locally finite perimeter and their reduced boundaries instead of properly embedded top-dimensional submanifolds and their boundaries. Standard regularity theory shows that the reduced boundary of a locally isoperimetric such set is smooth; see, e.g., the survey of results in \cite[Section 4]{eichmairmetzger}.
\end {rema}
Let $\Omega_1,\, \Omega_2, \ldots \subset M$ be locally isoperimetric regions of $(M, g)$ with $\partial \Omega_k \neq \emptyset$ for all $k$.
\begin {lem}  \label{area growth estimate}  There holds
\begin{align} \label{appendix area est} 
\sup_{k\geq 1} \, \sup_{r > 1} r^{1 - n} \, |B_r \cap \partial \Omega_k| < \infty.
\end{align} 
Moreover, for every $\alpha>0$ with $\alpha\neq n-1$,
\begin{align} \label{second estimate} 
\limsup_{k\to\infty}\sup_{1<s<t} \left(t^{n-1-\alpha}+s^{n-1-\alpha}\right)^{-1}\,\int_{(B_t\setminus B_{s})\cap  \partial \Omega_k}|x|_{\bar g}^{-\alpha}\,\mathrm{d}\bar\mu <\infty. 
\end{align} 
\end{lem}
\begin {proof}
Because $B_1\cup \partial B_1$ is compact, there exist compact regions $\tilde B_r\subset B_1$, $r\in(0,1)$, such that $|\tilde B_r|=r^n\,|B_1|$ and $\sup_{r\in(0,1)}|\partial \tilde B_r|<\infty$. For convenience, we let  $\tilde B_r=B_r$ if $r\geq 1$.

To prove \eqref{appendix area est},  let  $r>1$ be such that $\partial B_r$ and $\partial \Omega_k$ intersect transversely for all $k$. We may assume that $B_r\cap \partial \Omega_k\neq \emptyset$ so that $0<|B_r\cap \Omega_k|<|B_r|$. It follows that there is $\tilde r\in(0,r)$ such that $|\tilde B_{\tilde r}|=|B_r\cap \Omega_k|$. Let  $\tilde \Omega=\Omega\setminus B_r\cup \tilde B_{\tilde r}$. Note that $\tilde \Omega\,\triangle\, \Omega_k\Subset B_{2\,r}$ and  $|B_{2\,r}\cap\tilde \Omega|=|B_{2\,r}\cap \Omega_k|$. Moreover, we have $\partial \Omega_k\setminus \partial \tilde \Omega=B_r\cap \partial \Omega_k$ and $\partial \tilde \Omega\setminus \partial \Omega_k\subset \partial B_r\cup \partial \tilde B_{\tilde r}$. Since $\Omega_k$ is locally isoperimetric, it follows that
$$
|B_r\cap \partial \Omega_k|\leq |\partial B_r|+|\partial \tilde B_{\tilde r}|.
$$
Since $(M,g)$ is asymptotically flat, we have $$\sup_{r>1}\sup_{\tilde r\in(0,r)}r^{1-n}\,(|\partial B_r|+|\partial \tilde B_{\tilde r}|)<\infty.$$ This completes the proof of \eqref{appendix area est}.

To prove \eqref{second estimate}, we use \eqref{appendix area est} and Lemma \ref{layer cake}.
\end {proof}

\begin {lem}
Let $\Omega \subset M$ be a locally isoperimetric region with $\partial \Omega \neq \emptyset$. Then $\partial \Omega$ has constant mean curvature. The mean curvature is zero when the boundary is area-minimizing.
\end {lem}

\begin {lem} \label{lem:characterizationlocallyisoperimetric}
Let $ \Omega \subset \mathbb{R}^n$ be a locally isoperimetric region with $\partial  \Omega \neq \emptyset$. Then $\partial  \Omega$ is either a hyperplane or a coordinate sphere.
\begin {proof} 
This is \cite [Proposition 1]{Morgan-Ros:2010}. 
\end {proof}
\end {lem}

In Lemma \ref{finite perimter remark} and its proof, $\partial^*\Omega$ denotes the reduced boundary of a set $\Omega\subset \mathbb{R}^n$ of locally finite perimeter in $\mathbb{R}^n$.
\begin {lem} \label{finite perimter remark}
Let $\Omega\subset \mathbb{R}^n$ be a set of positive, locally finite perimeter in $\mathbb{R}^n$ that is locally isoperimetric in $\mathbb{R}^n \setminus \{0\}$.  Then $\partial^* \Omega$ is either a hyperplane or a coordinate sphere.
\end {lem}
\begin{proof}
	Let $\Omega \subset \mathbb{R}^n$ be a  set of locally finite perimeter in $\mathbb{R}^n$, $\partial^*\Omega\neq \emptyset$,  that is locally isoperimetric in $\mathbb{R}^n\setminus \{0\}$. We aim to show that $\Omega\subset \mathbb{R}^n$ is locally isoperimetric in $\mathbb{R}^n$.  The assertion of the lemma then follows from Remark \ref{regularity remark} and Lemma \ref{lem:characterizationlocallyisoperimetric}.
	
	 Let $\tilde \Omega\subset \mathbb{R}^n $ be another set of locally finite perimeter in $\mathbb{R}^n$ and suppose that there is $t>0$ such that $\Omega\,\triangle\, \tilde \Omega \Subset B^n_t(0)$ and $|B^n_t(0)\cap\tilde \Omega|_{\bar g}=|B^n_t(0)\cap \Omega|_{\bar g}$. Given $r\in(0,t/2)$, let 
	 $$\tilde \Omega_r=\{x\in \tilde \Omega:|x|_{\bar g}>r\}\cup \{x\in \Omega:|x|_{\bar g}<r\}.$$
	 Note that $\tilde \Omega_r$ is a set of locally finite perimeter in $\mathbb{R}^n$ with 
	 \begin{align} \label{contained}
	 \tilde \Omega_r\,\triangle\, \Omega\Subset \{x\in\mathbb{R}^n:r/2<|x|_{\bar g}<t\}. 
	 \end{align} 
	 Using that $|B^n_t(0)\cap \tilde \Omega|_{\bar g}=|B^n_t(0)\cap \Omega|_{\bar g}$,   we have
	\begin{equation} \label{volume difference} 
	\begin{aligned} 
	|\{x\in\mathbb{R}^n:r/2<|x|_{\bar g}<t\}\cap \tilde \Omega_r|_{\bar g}&=|\{x\in\mathbb{R}^n:r/2<|x|_{\bar g}<t\}\cap \tilde \Omega|_{\bar g}+o(1)\\&=|\{x\in\mathbb{R}^n:r/2<|x|_{\bar g}<t\}\cap  \Omega|_{\bar g}+o(1)
	\end{aligned} 
	\end{equation} 
as $r\to0$.	Using that $\Omega$ and $\tilde \Omega$ are sets of locally finite perimeter in $\mathbb{R}^n$, we obtain
	\begin{align} \label{perimeter difference} 
	|B^n_t(0)\cap \partial^* \tilde \Omega_r|_{\bar g}=|B^n_t(0)\cap \partial^* \tilde \Omega|_{\bar g}+o(1).
	\end{align} 
	 By \eqref{volume difference}, the argument in \cite[Proposition 3.1]{Morgan} shows that, enlarging $t>0$, if necessary,  for every $r>0$ sufficiently small, there is a set $\check \Omega_r\subset \mathbb{R}^n$ of locally finite perimeter in $\mathbb{R}^n$  with 
	\begin{itemize}
		\item[$\circ$] $\check \Omega_r\,\triangle\, \tilde \Omega_r \Subset \{x\in\mathbb{R}^n:2\,r<|x|_{\bar g}<t\}$,
		\item[$\circ$] $|\{x\in\mathbb{R}^n:r/2<|x|_{\bar g}<t\}\cap \check \Omega_r|_{\bar g}=|\{x\in\mathbb{R}^n:r/2<|x|_{\bar g}<t\}\cap \Omega|_{\bar g}$, and
		\item[$\circ$] $|\{x\in\mathbb{R}^n:r/2<|x|_{\bar g}<t\}\cap \partial^* \check \Omega_r|_{\bar g}=|\{x\in\mathbb{R}^n:r/2<|x|_{\bar g}<t\}\cap \partial^* \tilde \Omega_r|_{\bar g}+o(1)$.
	\end{itemize}
Using \eqref{contained}, we see that $\check \Omega_r\,\triangle\, \Omega\Subset \{x\in\mathbb{R}^n:r/2<|x|_{\bar g}<t\}$.
 Since $\Omega$ is locally isoperimetric in $\mathbb{R}^n\setminus \{0\}$, we have  
 \begin{align*}
 |\{x\in\mathbb{R}^n:r/2<|x|_{\bar g}<t\}\cap \partial^* \Omega|_{\bar g}&\leq |\{x\in\mathbb{R}^n:r/2<|x|_{\bar g}<t\}\cap \partial^* \check \Omega_{r}|_{\bar g}
 \\&=|\{x\in\mathbb{R}^n:r/2<|x|_{\bar g}<t\}\cap \partial^* \tilde \Omega_r|_{\bar g}+o(1).
 \end{align*}  Hence, $| B^n_t(0)\cap \partial^* \Omega|_{\bar g}\leq |B^n_t(0)\cap \partial^*  \tilde \Omega_r|_{\bar g}+o(1)$ as $r\to0$. Using \eqref{perimeter difference}, it follows that 
	$
	|B^n_t(0)\cap \partial^* \Omega|_{\bar g}\leq |B^n_t(0)\cap \partial^* \tilde \Omega|_{\bar g}
	$
	as asserted.
\end{proof}

Let $\Omega_1,\, \Omega_2, \ldots \subset M$ be locally isoperimetric regions of $(M, g)$ with $\partial \Omega_k \neq \emptyset$ for all $k$.

\begin {lem} \label{lem:locisoconverge}
Assume that $\limsup_{k\to\infty} |H (\partial \Omega_k)| < \infty$. Then $\limsup_{k\to\infty } |h (\partial \Omega_k)| < \infty$. Moreover, if there is $K\subset M$ compact such that $K\cap\partial M=\emptyset$ and $K\cap \partial \Omega_k\neq \emptyset$ for all $k$, then there exists a locally isoperimetric region $\Omega \subset M$ such that, passing to a subsequence, $\Omega_k \to \Omega$ locally smoothly.
\begin {proof}
Suppose, for a contradiction, that there are locally isoperimetric regions $\Omega_1,\, \Omega_2, \ldots \subset M$  with $\partial \Omega_k \neq \emptyset$ for all $k$ such that $\limsup_{k\to\infty} |H (\partial \Omega_k)| < \infty$ while $\limsup_{k\to\infty } |h (\partial \Omega_k)| = \infty$. By a standard rescaling argument, we obtain  Riemannian metrics $\tilde g_1,\,\tilde g_2,\ldots $  on $\mathbb{R}^n$ and regions  $\tilde \Omega_1,\,\tilde \Omega_2,\ldots\subset \mathbb{R}^n$ with the following properties. 
\begin{itemize}
	\item[$\circ$] $\tilde g_k$ converges locally smoothly to $\bar g$.
	\item[$\circ$] $0\in \partial \tilde \Omega_k$ for every $k$.
	\item[$\circ$] $\partial \tilde \Omega_k$ and $\partial B^n_k(0)$ intersect transversely for every $k$.
	\item[$\circ$]  For every $k$, there holds
	$$
	|B^n_k(0) \cap \partial \tilde \Omega_k |_{\tilde g_k} \leq |B^n_k(0) \cap \partial \check \Omega|_{\tilde g_k}
	$$
	for every region $\check \Omega \subset \mathbb{R}^n$ with $\check \Omega\,\triangle\,\tilde \Omega_k  \Subset B^n_k(0)$ and $|B^n_k(0) \cap \check  \Omega|_{\tilde g_k} = |B^n_k(0)\cap \tilde \Omega_k|_{\tilde g_k}$.
	\item[$\circ$]$H_{\tilde g_k}(\partial \tilde \Omega_k)=o(1)$ as $k\to\infty$.
	\item[$\circ$]  $\sup_{x\in B^n_k(0)} |h_{\tilde g_k} (\partial\tilde \Omega_k)(x)|_{\tilde g_k} = |h_{\tilde g_k} ( \partial \tilde \Omega_k)(0)|_{\tilde g_k}=1$ for every $k$. 
\end{itemize} 
Using Schauder estimates and standard convergence results from geometric measure theory, we see that $\tilde \Omega_k$ converges locally smoothly to a locally isoperimetric region $\tilde \Omega\subset \mathbb{R}^n$ with $\bar H(\partial \tilde \Omega)=0$ and $|\bar h(\partial \tilde \Omega)(0)|_{\bar g}=1$. This is not compatible with Lemma \ref{lem:characterizationlocallyisoperimetric}.
\end {proof}
\end {lem}

\begin {lem} \label{lem:locallyisoperimetricboundary}
Let $\Omega \subset M$ be a locally isoperimetric region with noncompact boundary. Then $\Omega$ has area-minimizing boundary. All components of $\partial \Omega$ except for one are components of $\partial M$.

\begin {proof}

Suppose, for a contradiction, that $H (\partial \Omega) \neq 0$. Let $x_1,\,x_2,\ldots \in \partial \Omega \setminus B_1$ be points with $|x_\ell|_{\bar g} \to \infty$. It follows from Lemma \ref{lem:locisoconverge} that, passing to a subsequence,
\[
- x_\ell + \Omega \setminus B_1  \to \check {\Omega} \text{ locally smoothly in } \mathbb{R}^n
\] 
where $\check \Omega \subset \mathbb{R}^n$ is a locally isoperimetric region with $\partial \check \Omega \neq \emptyset$ and $\bar H (\partial \check \Omega) = H (\partial \Omega)$. By Lemma \ref{lem:characterizationlocallyisoperimetric}, $\partial \check \Omega$ is a coordinate sphere of radius $(n-1) \, |H(\partial \Omega)|^{-1}$. It follows that $\Omega$ has infinitely many bounded components, each one close to a coordinate ball of radius $(n-1) \, |H (\partial \Omega)|$.  Such a configuration is not compatible with the Euclidean isoperimetric inequality. Thus $H (\partial \Omega) = 0$.

Let $s_1,\,s_2,\ldots >1$ be numbers with $s_\ell \to \infty$. By Lemma \ref{lem:locisoconverge}, passing to a subsequence,
\[
s_\ell^{-1} \, (\Omega \setminus B_1) \to \breve \Omega \text{ locally smoothly in } \mathbb{R}^n \setminus \{0\}
\]
where $\breve \Omega \subset \mathbb{R}^n$ is a locally isoperimetric region whose boundary is nonempty with mean curvature zero. By Lemma \ref{finite perimter remark}, $\breve \Omega$ is a half-space. It follows that $\partial \Omega$ has only one unbounded component. To see that $\partial \Omega$ is area-minimizing, observe that there are constants $\delta, c > 0$ such that, for every $\ell$ sufficiently large and all $v \in (- \delta, \delta),$ there is a region $\tilde \Omega \subset M$ with $\tilde \Omega \,  \triangle \, \Omega \Subset \{x \in \mathbb{R}^n : s_\ell < |x|_{\bar g} < 2 \, s_\ell\} = U_\ell$ and such that
\begin{eqnarray*} 
	s_\ell^{- n} \, \big(  | U_\ell \cap \Omega| - |U_\ell \cap \tilde \Omega| \big) = v  \qquad \text{ and } \qquad s_\ell^{1-n} \, \big| |U_\ell \cap \partial \Omega| - |U_\ell \cap \partial \tilde \Omega | \big| \leq c \, v^2;
	\end {eqnarray*}
	see also \cite[Proposition 3.1]{Morgan}. 	Thus, we can   add and subtract an amount of volume $V$ from $\Omega$ at the cost of changing the boundary area by an amount of order $o(1)$. It follows that $\Omega$ has area-minimizing boundary.
	
	The preceding argument also shows that, in the case where $\partial M=\emptyset$, $\partial \Omega$ has no bounded components, since deleting such components and compensating for the loss of volume far out would decrease area. In the case where $\partial M\neq \emptyset$, bounded components of $\partial \Omega$, being closed minimal surfaces, are contained in $\partial M$.
	\end {proof}
	\end {lem}
	
	We now assume that $\Omega_1,\,\Omega_2,\ldots \subset M$ are isoperimetric regions of $(M,g)$ with $|\Omega_k| \to \infty$. Let
	$$
	 \lambda(\partial \Omega_k)=\left(\frac{|\partial \Omega_k|}{n\,\omega_{n}}\right)^{\frac{1}{n-1}}
	$$
	be the area radius of $\partial \Omega_k$. 
	\begin {lem} \label{lem:coarseisobound} 
	There holds
	\begin {eqnarray}
	\lim_{k \to \infty}  \lambda (\partial \Omega_k)^{-n} \, |\Omega_k| = \omega_n.
	\end {eqnarray}
	\begin {proof} 
	By comparison with coordinate balls far out in the asymptotically flat end, we see that
	\[
	\liminf_{k \to \infty}  \lambda (\partial \Omega_k)^{-n} \, |\Omega_k| \geq \omega_n.
	\] 
\indent Let $\varepsilon>0$ and $r>1$ be large such that $\partial B_r$ and $\partial \Omega_k$ intersect  transversely for all $k$ sufficiently large. Let $\Omega_{k,r}=\Omega_k\setminus B_r$. Note that 
	\begin{align}\label{cut and paste comparison}
	| \Omega_{k,r}|\geq |\Omega_k|-|B_r|\qquad \text{and}\qquad |\partial \Omega_k|\geq |\partial  \Omega_{k,r}|-|\partial B_r|.  
	\end{align} 
	By the Euclidean isoperimetric inequality,
	$$
	\bar \lambda (\partial \Omega_{k,r})^{-n}\,|\Omega_{k,r}|_{\bar g}\leq \omega_n.
	$$
	Increasing $r>1$ if necessary, we obtain
	$$
	 \lambda (\partial \Omega_{k,r})^{-n}\,|\Omega_{k,r}|\leq \omega_n+\varepsilon
	$$ 
	for all $k$ sufficiently large. Letting $k\to\infty$ and using \eqref{cut and paste comparison}, we conclude that
	\[
	\limsup_{k \to \infty}  \lambda (\partial \Omega_k)^{-n} \, |\Omega_k| \leq \omega_n+\varepsilon.
	\] 
	The assertion follows. 
	\end {proof}
	\end {lem}
	
	It follows from Lemma \ref{lem:coarseisobound} that $\lambda (\partial \Omega_k) \to \infty$.
	
	
	From now on, we assume that $H (\partial \Omega_k) > 0$ for all $k$.
	
	\begin {lem} \label{lem:notused}
	There holds $\limsup_{k \to \infty} \lambda (\partial \Omega_k) \, H(\partial \Omega_k) < \infty$.
	\begin {proof}
	Suppose, for a contradiction, that $\lambda (\partial \Omega_k) \, H (\partial \Omega_k) \to \infty$. By  Lemma \ref{lem:coarseisobound}, there are points $x_1,\,x_2,\ldots \in \partial \Omega_k$ such that $\liminf_{k \to\infty} \lambda (\partial \Omega_k)^{-1} \, |x_k|_{\bar g} >0$. Let $r_k = (n-1) \, H(\partial \Omega_k)^{-1}$. By Lemma \ref{lem:characterizationlocallyisoperimetric} and   Lemma \ref{lem:locisoconverge}, passing to a subsequence, there is $\xi\in\mathbb{R}^n$ such that 
	\[
	r_k^{-1} \, (- x_k + \Omega_k \setminus B_1) \to \{x \in \mathbb{R}^n : |x-\xi|_{\bar g} \leq 1\} \text{ locally smoothly in } \mathbb{R}^n. 
	\]
It follows that the component of $\Omega_k$ which contains $x_k$ is close to a coordinate ball of radius $r_k = o (\lambda (\partial \Omega_k))$. Using Lemma \ref{lem:coarseisobound}, we see that $\Omega_k$ contains a second such component.  
Such a configuration is not compatible with the Euclidean isoperimetric inequality. 
	\end {proof}
	\end {lem}
	
	\begin {lem}  \label{large blowdown1}
	Assume that there are points $x_1,\,x_2,\ldots \in \partial \Omega_k \setminus B_1$ with $\lambda (\partial \Omega_k)^{-1} \, |x_k|_{\bar g} \to \infty$. The component of $\Omega_k$ that contains $x_k$ is close to a coordinate ball of radius $(n-1)\,H(\partial \Omega_k)^{-1}$.
	\end {lem}
	\begin{proof}
Let $\Omega^0_k\subset \Omega_k$ be the component that contains $x_k$.	Using Lemma \ref{lem:characterizationlocallyisoperimetric}, Lemma  \ref{lem:locisoconverge}, and Lemma \ref{lem:notused}, we see that $\lambda(\partial \Omega_k)^{-1}(-x_k+\Omega^0_k\setminus B_1)$ converges locally smoothly to a Euclidean ball as asserted.
	\end{proof}
	It follows from Lemma \ref{large blowdown1} that $\Omega_k$ has at most one component that, on the scale of $\lambda (\partial \Omega_k)$, is far from $B_1$.
	
	\begin {lem}  \label{large blowdown2}
	Assume that there are points $x_1,\,x_2,\ldots \in \partial \Omega_k \setminus B_1$ with $0 < \liminf_{k \to \infty} \lambda (\partial \Omega_k)^{-1} \, |x_k|_{\bar g} <\infty$. There is $\xi \in \mathbb{R}^n$ such that, passing to a subsequence,
	\[
	\lambda (\partial \Omega_k)^{-1} \,( \Omega_k \setminus B_1) \to \{x \in \mathbb{R}^n : |x - \xi|_{\bar g} \leq 1\}  \text{ locally smoothly in } \mathbb{R}^n \setminus \{0\}.
	\]
	\begin {proof} 
	By Lemma \ref{area growth estimate}, Lemma  \ref{lem:locisoconverge}, and Lemma \ref{lem:notused}, passing to a subsequence, 	$\lambda (\partial \Omega_k)^{-1} \,( \Omega_k \setminus B_1)$ converges locally smoothly in $\mathbb{R}^n\setminus \{0\}$  to a  set of positive, finite perimeter in $\mathbb{R}^n$ that is locally isoperimetric in $\mathbb{R}^n\setminus \{0\}$, the area of whose reduced boundary is at most $n \, \omega_n$. According to Lemma \ref{finite perimter remark}, such a set is a ball of radius $0 < r \leq 1$. Note that 
	$
	(n-1) \, / \, r = \lim_{k \to \infty} \lambda (\partial \Omega_k) \, H (\partial \Omega_k)
	$.
	Suppose, for a contradiction, that $r < 1$. Using Lemma \ref{lem:coarseisobound} and Lemma \ref{large blowdown1}, we see that $\Omega_k$ contains  at least one additional large component that lies far out. Such a configuration is not compatible with the Euclidean isoperimetric inequality. 
	\end {proof}
	\end {lem}
	
	\begin {coro} 
	There holds 
	\begin {eqnarray} \label{H est}
	\lim_{k \to \infty} \lambda (\partial \Omega_k) \, H (\partial \Omega_k) = n-1.
	\end {eqnarray}
	\end {coro}
	\begin{proof}
		This follows from Lemma \ref{large blowdown2}.
	\end{proof}
	Now, we assume in addition that there is $K \subset M$ compact such that $K\cap \partial M$ and, for all $k$,
	\[
	K \cap \partial \Omega_k \neq \emptyset.
	\]	
	\begin {prop} \label{smooth local convergence}
	There is a region $\Omega \subset M$ with noncompact area-minimizing boundary such that, passing to a subsequence, $\Omega_k \to \Omega$ locally smoothly. There is $\xi \in \mathbb{R}^n$ with $|\xi|_{\bar g} = 1$ such that, passing to a subsequence, $\lambda (\partial \Omega_k)^{-1} \, (\Omega_k \setminus B_1) \to \{ x \in \mathbb{R}^n : |x - \xi|_{\bar g} \leq 1\}$ locally smoothly in $\mathbb{R}^n \setminus \{0\}$.
	\begin {proof}
	This follows from Lemma \ref{lem:locisoconverge}, Lemma \ref{lem:locallyisoperimetricboundary}, and Lemma \ref{large blowdown2}.
	\end {proof}
	\end {prop}
	
	In the following lemma, $\hcirc$ denotes the traceless second fundamental form. 
	
	\begin {lem}
	Let $x_1,\,x_2,\ldots \in \partial \Omega_k \setminus B_1$ be points with $|x_k|_{\bar g} \to \infty$. Then 
	\begin{align} \label{hcirc est} 
	\limsup_{k\to\infty} |x_k|_{\bar g} \, | \hcirc (\partial \Omega_k) (x_k)| =0. 
	\end{align}
	\begin {proof}
	If  $\lim_{k \to \infty}  \lambda (\partial \Omega_k)^{-1} \,|x_k|_{\bar g} = 0$, then, by Lemma \ref{finite perimter remark} and Lemma \ref{lem:locisoconverge}, passing to a subsequence, 
	$|x_k|_{\bar g}^{-1} \, ( \Omega_k \setminus B_1)$ converges to a half-space locally smoothly in $\mathbb{R}^n\setminus\{0\}$.
	If $\liminf_{k \to\infty} \, \lambda (\partial \Omega_k)^{-1} |x_k|_{\bar g} >0$, then, by  Proposition \ref{smooth local convergence}, passing to a subsequence, $|x_k|_{\bar g}^{-1} \, (\Omega_k \setminus B_1)$ converges to a ball locally smoothly in $\mathbb{R}^n\setminus \{0\}$. Either way, the assertion follows.
	\end {proof}
	\end {lem}
	
	The isoperimetric profile of $(M, g)$ is the function  $A : (0, \infty) \to (0, \infty)$ given by
	\begin{align}  \label{isoperimetric profile} 
	A(V) = \inf \{ |\partial \Omega| : \Omega \subset M \text{ is a compact region with } |\Omega| = V \}.
	\end{align}
	
	\begin {lem} \label{profile properties} The isoperimetric profile of $(M, g)$ is absolutely continuous. As $V \to \infty$,
	\[
	\left(\omega_n^{-1} \, V\right)^{\frac{1-n}{n}}\,A(V)=n\,\omega_n+o(1).
	\]
For almost every $V>0$, as $V\to\infty$,
	\[
	\left(\omega_n^{-1} \, V\right)^{\frac{1}{n}}\,A'(V)=(n-1)+o(1).
	\]
	If $\partial M\neq \emptyset$, then $A$ is  strictly increasing.
	\begin {proof}
	See, e.g., \cite[Appendix A]{CESH} and \cite[Proposition 4]{eichmairmetzger}.
	\end {proof}
	\end {lem}
	
	\begin {lem} [{\cite[Theorem 1.12]{CCE}}]\label{existence iso}
	Assume that $(M, g)$ has positive mass. For every sufficiently large amount of volume $V>0$, there exists an isoperimetric region $\Omega \subset M$ with $|\Omega| = V$.
	\end {lem}

	\section{Variation of area and volume}
	\label{stable cmc} 
	In this section, we recall  the first and second variational formulae for area and volume and the definition of stable constant mean curvature surfaces; see, e.g., \cite[Appendix H]{CCE}. \\ \indent 
	Let $(M,g)$ be a Riemannian manifold without boundary of dimension $n\geq 3$.  Let $\Sigma\subset M$ be a closed hypersurface bounding a compact region $\Omega$. We denote by $\mathrm{d}\mu$ the area element, by $\nu$ the outward pointing unit normal, and by $h$ and $H$ the second fundamental form and the mean curvature of $\Sigma$, respectively, computed with respect to $\nu$. \\ \indent
	Let $\varepsilon>0$ and $U\in C^\infty(\Sigma\times(-\varepsilon,\varepsilon))$ with $U(x,0)=0$ for all $x\in \Sigma$. Decreasing $\varepsilon$ if necessary, we obtain a smooth family $\{\Sigma(s):s\in(-\varepsilon,\varepsilon)\}$ of hypersurfaces $\Sigma(s)\subset M$ where 
	$$
\Sigma(s)=\{\exp_x(U(x,s)\,\nu(x)):x\in \Sigma\}.
	$$
	We define the initial velocity $u\in C^\infty(\Sigma)$ and the initial acceleration $v\in C^\infty(\Sigma)$ by
	$$
	u(x)=\dot{U}(x,0)\qquad\text{and}\qquad v=\ddot{U}(x,0).
	$$  \indent 
	Let $\Omega(s)$ be the compact region bounded by $\Sigma(s)$.
	\begin{lem} \label{variation of area and volume} 
		There holds
		\begin{align} 
\notag		\frac{d}{ds}\bigg|_{s=0}|\Sigma(s)|&=\int_{\Sigma}H\,u\,\mathrm{d}\mu,\\
	\label{second area} 	\frac{d^2}{ds^2}\bigg|_{s=0}|\Sigma(s)|&=\int_{\Sigma}H\,v+H^2\,u^2+|\nabla u|^2-(|h|^2+{Ric}(\nu,\nu))\,u^2\,\mathrm{d}\mu.
		\end{align} 
	Moreover,
	\begin{align*} 
		\frac{d}{ds}\bigg|_{s=0}|\Omega(s)|&=\int_{\Sigma}u\,\mathrm{d}\mu,\\
		\frac{d^2}{ds^2}\bigg|_{s=0}|\Omega(s)|&=\int_{\Sigma}v+H\,u^2\,\mathrm{d}\mu.
	\end{align*} 
	\end{lem}
Assume that for every such variation  satisfying also
$$
\frac{d}{ds}\bigg|_{s=0}|\Omega(s)|=\frac{d^2}{ds^2}\bigg|_{s=0}|\Omega(s)|=0,
$$
there holds
$$
\frac{d}{ds}\bigg|_{s=0}|\Sigma(s)|=0\qquad\text{and}\qquad \frac{d^2}{ds^2}\bigg|_{s=0}|\Sigma(s)|\geq 0.
$$
Note that the mean curvature $H$ of  $\Sigma$ is constant in this case. We say that $\Sigma$ is a stable constant mean curvature surface.\\ \indent 
Note that each component of the boundary of an isoperimetric region $\Omega\subset M$ is a stable constant mean curvature surface with the same constant mean curvature.
	      \\ \indent 
	We record the following integration by parts formula for the second variation of area formula \eqref{second area} with respect to a Euclidean translation. 
	\begin{lem} \label{euclidean lemma}
Let $ \Sigma\subset \mathbb{R}^n$ be a closed oriented hypersurface with boundary $\partial  \Sigma$. Let $\bar\nu$ be a unit normal of $\Sigma$ and $\bar\omega$ the outward-pointing conormal of $\partial \Sigma$. Let  $\bar u, \bar v:\Sigma\to\mathbb{R}$ be given by 
$$
\bar u=\bar g(e_n,\bar\nu)\qquad\text{and}\qquad \bar v=-\bar h(e_n^{\bar\top},e_n^{\bar\top}).
$$
 There holds 
		$$
		\int_{ \Sigma}\bar H\,\bar v+\bar H^2\,\bar u^2+|\bar\nabla\bar u|_{\bar g}^2-|\bar h|^2_{\bar g}\,\bar u^2\,\mathrm{d}\bar \mu=\int_{\partial \Sigma}\bar h(\bar\omega,e_n^{\bar\top})\,\bar u\,\mathrm{d}\bar{l}-\int_{ \partial \Sigma}\bar g(e_n^{\bar\top},\bar\omega) \bar H\,\bar u\,\mathrm{d}\bar{l}.
		$$
	\end{lem}
	\begin{proof}
		Let $X,\,Y$ be tangent fields along $\Sigma$. There holds 
		\begin{align*}
			&\circ\qquad 	\bar\nabla_X u=\bar h(e_n^{\bar \top},X)\text{ and}\hspace{8cm} 
		\\&\circ\qquad  \bar g(\bar \nabla_X e_n^{\bar \top},Y)=-\bar h(X,Y)\,\bar u. 
		\end{align*}
		It follows that $\bar \Delta_\Sigma u=(\bar{ \operatorname{div}}_\Sigma\, \bar h)(e_n^\top)-|\bar h|_{\bar g}^2\,\bar u$. Multiplying by $\bar u$ and integrating by parts,  we obtain
		$$
		\int_{\Sigma} |\bar\nabla u|_{\bar g}^2\,\mathrm{d}\bar \mu=\int_{\partial \Sigma}\bar h(\bar\omega,e_n^{\bar\top})\,\bar u\,\mathrm{d}\bar{l}+\int_{\Sigma} |\bar h|^2_{\bar g}\,\bar u^2-(\bar{ \operatorname{div}}_\Sigma\, \bar h)(e_n^\top)\,\bar u\,\mathrm{d}\bar \mu.
		$$
		Using the Codazzi equation $\bar\nabla_{e_n^{\bar\top}}\bar H=(\bar{\operatorname{div}}\,\bar h)(e_n^{\bar\top})$, we have 
		$$
		\bar {\operatorname{div}}_{\Sigma}(\bar H\,\bar u\, e_n^{\bar \top})=(\bar{\operatorname{div}}_\Sigma\bar h)(e_n^{\bar \top})-\bar H^2\,\bar u^2-\bar H\,\bar v.
		$$
		 Using integrating by parts again, we obtain
		$$
		\int_{ \Sigma} (\bar{ \operatorname{div}}_\Sigma\, \bar h)(e_n^\top)\,\bar u\,\mathrm{d}\bar \mu=\int_{ \partial \Sigma}\bar g(e_n^{\bar\top},\bar\omega) \bar H\,\bar u\,\mathrm{d}\bar{l}+\int_{ \Sigma} \bar H^2\,\bar u^2+\bar H\,\bar v\,\mathrm{d}\bar \mu.
		$$
	\end{proof}
	
		\section{Geometry of hypersurfaces in an asymptotically flat end}
	In this section, we assume that $g$ is a Riemannian metric on $\mathbb{R}^n$ where $n\geq 3$ and that, for some $\tau>0$, $$|g-\bar g|+|x|_{\bar g}\,|\bar D g|_{\bar g}+|x|^2_{\bar g}\,|\bar D^2 g|_{\bar g}=O(|x|_{\bar g}^{-\tau}).$$
	Let $\Sigma\subset \mathbb{R}^n$ be a two-sided  hypersurface with area element $\mathrm{d}\mu$, designated normal $\nu$, and second fundamental form $h$ and mean curvature $H$ with respect to $\nu$.  The corresponding Euclidean quantities are denoted with a bar.
	\begin{lem} As $x\to\infty$,
		\label{geometric expansions}
		\begin{align*}
			\nu=&\,\bar\nu+O(|x|_{\bar g}^{-\tau}),\\
			\mathrm{d}\mu=&\,(1+O(|x|_{\bar g}^{-\tau}))\,\mathrm{d}\bar\mu,
			\\	|x|_{\bar g}\,h=&|x|_{\bar g}\,\bar h +O(|x|_{\bar g}^{-\tau})+O(|x|_{\bar g}^{1-\tau}\,|\bar h|_{\bar g}),
			\\	|x|_{\bar g}\,H=&|x|_{\bar g}\,\bar H +O(|x|_{\bar g}^{-\tau})+O(|x|_{\bar g}^{1-\tau}\,|\bar h|_{\bar g}),\\
			|x|^2_{\bar g}\,\nabla h=&|x|^2_{\bar g}\,\bar\nabla \bar h +O(|x|_{\bar g}^{-\tau})+O(|x|_{\bar g}^{1-\tau}\,|\bar h|_{\bar g})+O(|x|_{\bar g}^{2-\tau}\,|\bar\nabla \bar h|_{\bar g}),\text{ and}
			\\	|x|^2_{\bar g}\,\nabla H=&|x|^2_{\bar g}\,\bar\nabla \bar H +O(|x|_{\bar g}^{-\tau})+O(|x|_{\bar g}^{1-\tau}\,|\bar h|_{\bar g})+O(|x|_{\bar g}^{2-\tau}\,|\bar\nabla \bar h|_{\bar g}).
		\end{align*}
	\end{lem}
	\begin{proof} We give a proof of the first and third asserted estimate. 
		Given $t\in[0,1]$, let $g_t=\bar g+t\,\sigma$ where $\sigma=g-\bar g$. Note that $g_0=\bar g$ and $g_1=g$.  We use a dot to denote the linearization of a geometric quantity at $t=0$. 	  Given $p\in \Sigma$, let $\{f_1,\ldots,f_{n-1}\}$ be a  basis of $T_p\Sigma$ that is orthonormal with respect to $\bar g$. 
		
		Using that $g_t(\nu_{g_t},\nu_{g_t})=1$ and $g_t(\nu_{g_t},f_1)=\ldots=g_t(\nu_{g_t},f_{n-1})=0$, we see that 
		\begin{align} \label{dot nu} 
		\dot \nu=-\frac12\,\sigma(\bar \nu,\bar\nu)\,\bar\nu-\sum_{i=1}^{n-1}\sigma(\bar \nu,f_i)\,f_i=O(|x|_{\bar g}^{-\tau}).
		\end{align} 
		
		Let $a,\,b=1,\ldots,n-1$. Using that $h_{g_t}(f_a,f_b)=g_t({\nabla_{g_t}}_{f_a}\nu_{g_t},f_b)$, we obtain that 
		$$
		\dot h(f_a,f_b)=\sigma(\bar \nabla_{f_a} \bar\nu,f_b)+\bar g(\dot \nabla_{f_a} \bar \nu, f_b)+\bar g(\bar \nabla_{f_a}\dot \nu,f_b).
		$$
		In conjunction with \eqref{dot nu} and the estimates $\bar \nabla_{f_a}\bar\nu=O(|\bar h|_{\bar g})$,\, $\bar \nabla_{f_a} f_b=O(|\bar h|_{\bar g})$, and $\dot \nabla_{f_a}\bar \nu=O(|x|_{\bar g}^{-1-\tau})$, we see that 
		$$
		|x|_{\bar g}\,\dot h=O(|x|_{\bar g}^{-\tau})+O(|x|_{\bar g}^{1-\tau}\,|\bar h|_{\bar g}).
		$$  
		
		The assertion follows from these estimates and Taylor's theorem. 
	\end{proof}
	\begin{lem} \label{layer cake}
		Suppose that $\partial \Sigma=\emptyset$. 	Let  $\alpha>0$ with $\alpha\neq n-1$, $0<s<t,$ and suppose that, for some $c\geq1$,
		$$
		r^{1-n}\,	|B^n_{r}  (0)\cap  \Sigma|_{\bar g}\leq c
		$$
		for all $s<r<t$.  There holds
		$$
	\int_{(B^n_t(0)\setminus B^n_{s}(0))\cap  \Sigma}|x|_{\bar g}^{-\alpha}\,\mathrm{d}\bar\mu\leq c\,t^{n-1-\alpha}+\frac{c\,\alpha}{n-1-\alpha}\,\left(t^{n-1-\alpha}-s^{n-1-\alpha}\right).
		$$
	\end{lem}
\begin{proof}
	Using the layer cake representation for $|x|_{\bar g}^{-\alpha}$, we see that
	$$
		\int_{(B^n_t(0)\setminus B^n_{s}(0))\cap  \Sigma}|x|_{\bar g}^{-\alpha}\,\mathrm{d}\bar\mu=\alpha \,\int_{0}^\infty r^{-\alpha-1}\,|B^n_r(0)\cap(B^n_t(0)\setminus B^n_{s}(0))\cap  \Sigma |_{\bar g}\, \mathrm{d}r.
	$$
	Note that 
	\begin{align*} 
	&\int_{0}^\infty r^{-\alpha-1}\,|B^n_r(0)\cap(B^n_t(0)\setminus B^n_{s}(0))\cap  \Sigma |_{\bar g}\, \mathrm{d}r
	\\&\qquad= \int_{t}^\infty r^{-\alpha-1}\,|(B^n_t(0)\setminus B^n_{s}(0))\cap \Sigma|_{\bar g}\, \mathrm{d}r+\int_s^tr^{-\alpha-1}\,|(B^n_r(0)\setminus B^n_{s}(0))\cap \Sigma|_{\bar g}\,\mathrm{d}r.
	\end{align*} 
The assertion follows, using the estimates
$$
 \int_{t}^\infty r^{-\alpha-1}\,|(B^n_t(0)\setminus B^n_{s}(0))\cap \Sigma|_{\bar g}\, \mathrm{d}r=\alpha^{-1}\,t^{-\alpha}\,|(B^n_t(0)\setminus B^n_{s}(0))\cap \Sigma|_{\bar g}\leq c\,\alpha^{-1}\,t^{n-1-\alpha}
$$
and 
$$
\int_s^tr^{-\alpha-1}\,|(B^n_r(0)\setminus B^n_{s}(0))\cap \Sigma|_{\bar g}\,\mathrm{d}r\leq c\,\int_s^t\,r^{n-2-\alpha}\,\mathrm{d}r=\frac{c\,\alpha}{n-1-\alpha}\,(t^{n-1-\alpha}-s^{n-1-\alpha }).
$$
\end{proof}

\section{A Liouville theorem on the slab}
In this section, we prove a Liouville theorem for harmonic functions on a slab.
\begin{lem} \label{harmonic on slab} 
Let $n\geq 2$.	Let $f\in C^\infty(\mathbb{R}^{n-1}\times[0,2])$ be a nonnegative harmonic function with $f(x)=0$ for all $x\in\mathbb{R}^{n-1}\times \{0,\,2\}$.  Assume that 
	\begin{align} \label{bound} \sup\{|(\bar D f)(x)|_{\bar g}:x\in\mathbb{R}^{n-1}\times\{0\}\}<\infty.\end{align}
	Then $f=0$. 
\end{lem}
\begin{proof}
	Let $x_0\in \mathbb{R}^{n-1}\times\{1\}$ and $v\in C^\infty(\mathbb{R}^{n-1}\times[0,2])$ be given by $v(x)=f(x_0)\,x^n$. Note that $v$ is harmonic. By the Boundary Harnack comparison principle, see, e.g., \cite[Theorem 11.6]{caffarelli}, there is a constant $c>0$ depending only on $n$ such that $v\leq c\,f$ on $B^{n}_1(x_0-e_n)\cap (\mathbb{R}^{n-1}\times[0,2])$. In particular,
	$$
	f(x_0)=(\partial_{e_n} v)(x_0-e_n)\leq c\,(\partial_{e_n} f)(x_0-e_n).
	$$
	In conjunction with \eqref{bound}, we see that $f$ is bounded on $\mathbb{R}^{n-1}\times\{1\}$. By the Boundary Harnack inequality, see, e.g., \cite[Theorem 11.5]{caffarelli}, it follows that $f$ is bounded in all of $\mathbb{R}^{n-1}\times[0,2]$. Let $\tilde f\in C^0(\mathbb{R}^n)$ be the unique function satisfying 
	\begin{align*} 
&\circ\qquad	\tilde f(x^1,\ldots,\,x^{n-1},\,x^n+4\,k)=f(x^1,\ldots,x^{n-1},\,x^n)\text{ and}\\
&\circ \qquad 	\tilde f(x^1,\ldots,\,x^{n-1},\,x^n+2+4\,k)=-f(x^1,\ldots,x^{n-1},\,x^n)
	\end{align*} 
	for all $x\in \mathbb{R}^{n-1}\times[0,2]$ and every integer $k$. Note that $\tilde f$ is bounded. According to the Schwarz reflection principle,
	 $\tilde f\in C^\infty(\mathbb{R}^{n})$ and $\tilde f$ is harmonic.  By the Liouville theorem, $\tilde f$ is  constant. The assertion follows.
\end{proof}
\begin{rema}
	The function $f\in C^\infty(\mathbb{R}^{n-1}\times[0,2])$ given by $f(x)=\exp(x^1/(2\,\pi))\,\sin(x^n/(2\,\pi))$ is nonnegative, harmonic, satisfies $f(x)=0$ for all $x\in\mathbb{R}^{n-1}\times\{0,\,2\},$ and does not fulfill \eqref{bound}.
\end{rema}

\end{appendices}

\end{document}